\documentclass[10pt]{article}
\author{Sabrina Pauli}
\date{}
\title{Quadratic types and the dynamic Euler number of lines on a quintic threefold}
\usepackage[left=2cm, right=5cm, top=2cm]{geometry}
\usepackage{amsmath}
\usepackage{amssymb}
\usepackage{amsthm}
\usepackage[all]{xy}
\usepackage{amscd}
\usepackage{tikz-cd}
\usepackage{cite}
\usepackage{enumitem}
\usepackage{fullpage}
\usepackage{verbatim}
\usepackage[onehalfspacing]{setspace}
\usepackage{hyperref}

\usepackage{mathtools}
\usepackage{color,soul}
\usepackage{graphicx}
\usepackage{adjustbox}

\newtheorem{theorem}{Theorem}[section]
\newtheorem{proposition}[theorem]{Proposition}
\newtheorem{corollary}[theorem]{Corollary}
\newtheorem{remark}[theorem]{Remark}
\newtheorem{lemma}[theorem]{Lemma}
\newtheorem{claim}[theorem]{Claim}

\theoremstyle{definition}
\newtheorem{definition}[theorem]{Definition}
\newtheorem{example}[theorem]{Example}

\usepackage[utf8]{inputenc}

\newcommand{\C}{\mathbb{C}}
\newcommand{\Z}{\mathbb{Z}}
\newcommand{\A}{\mathbb{A}}
\newcommand{\PP}{\mathbb{P}}
\newcommand{\M}{\mathcal{M}}
\newcommand{\RR}{\mathbb{R}}

\newcommand{\Res}{\operatorname{Res}}
\newcommand{\Gr}{\operatorname{Gr}}
\newcommand{\Sym}{\operatorname{Sym}}
\newcommand{\GW}{\operatorname{GW}}
\newcommand{\Tr}{\operatorname{Tr}}

\newcommand{\Spec}{\operatorname{Spec}}
\newcommand{\HH}{\mathbb{H}}
\newcommand{\mathcolorbox}[2]{\colorbox{#1}{$\displaystyle #2$}}
\newcommand{\ind}{\operatorname{ind}}
\newcommand{\CH}{\operatorname{CH}}
\newcommand{\tCH}{\widetilde{\operatorname{CH}}}

\usepackage{xcolor}

\usepackage[
textwidth=3cm,
textsize=small,
colorinlistoftodos]
{todonotes}

\begin{document}
\maketitle
\begin{abstract}

We provide a geometric interpretation of the local contribution of a line to the count of lines on a quintic threefold over a field $k$ of characteristic not equal to 2, that is, we define the \emph{type} of a line on a quintic threefold and show that it coincides with the local index at the corresponding zero of the section of $\Sym^5\mathcal{S}^*\rightarrow\Gr(2,5)$ defined by the threefold.

Furthermore, we define the \emph{dynamic Euler number} which allows us to compute the $\A^1$-Euler number as the sum of local contributions of zeros of a section with non-isolated zeros which deform with a general deformation. As an example we provide a quadratic count of 2875 distinguished lines on the Fermat quintic threefold which computes the dynamic Euler number of $\Sym^5\mathcal{S}^*\rightarrow\Gr(2,5)$.

Combining those two results we get that when $k$ is a field of characteristic not equal to 2 or 5
\[\sum\Tr_{k(l)/k}(\operatorname{Type}(l))=1445\langle1\rangle+1430\langle-1\rangle\in\GW(k)\]
where the sum runs over the lines on a general quintic threefold.
\end{abstract}
\section{Introduction}
On a general quintic threefold there are finitely many lines. While the number of complex lines on a general quintic threefold is always $2875$, the number of real lines depends on the choice of quintic threefold. 
However, we get an invariant count, namely 15, if we count each line with an assigned sign \cite{MR3176620,MR3370064}.
So the real lines on a general quintic threefold can be divided into two `types', one contributes a $(+1)$, the other a $(-1)$ to the invariant signed count.
In \cite{MR4230388} Finashin and Kharlamov provide a geometric interpretation of these assigned signs. They generalize the definition of the (Segre) type of a real line on a cubic surface defined by Segre \cite{MR0008171} to the type of a real line on a degree $(2n-1)$-hypersurface in $\PP^{n+1}$ and define it to be the product of \emph{degrees} of certain \emph{Segre involutions}. In particular, they define the \emph{type} of a real line $l$ on a quintic threefold $X$ as follows: Any pair of points $r,s$ on the line with the same tangent space $T:=T_sX=T_rX$ in $X$, defines an involution $i:l\cong\PP^1_{\RR}\rightarrow \PP^1_{\RR}\cong l$ which sends $p\in l$ to the unique point $q$ with $T\cap T_pX=T\cap T_qX$. To each involution they assign a $(+1)$ if the involution has fixed points defined over $\RR$ and a $(-1)$ if it does not. The type of a real line is the product of the assigned signs. 
Their definition naturally generalizes to lines defined over a field $k$ of characteristic not equal to 2: 
Let $l$ be a line defined over $k$ on a general quintic threefold $X$. Then there are 3 pairs of points on $l$ which have the same tangent space. Let $r,s$ be one of those pairs, that is $T:=T_sX=T_rX$, and assume that the closed subscheme of $r$ and $s$ is defined over a finite field extension $L$ of $k$. Then we get an involution $i:l_L\rightarrow l_L$ of the base change of $l$ to $L$ that sends a point $p\in l_L$ to the point $q\in l_L$ with $T\cap T_pX=T\cap T_qX$. This involution has fixed points defined over $L(\sqrt{\alpha})$ for some $\alpha\in L^{\times}/(L^{\times})^2$. We say $\alpha$ is the degree of the involution $i$. Choosing suitable reprentatives of the degrees of the involutions, the product of the three degrees yields a well-defined element in $\GW(k)$.
\begin{definition}
The \emph{type} of the line $l$ is the product of the degrees of the 3 involutions viewed as an element of $\GW(k)$.
\end{definition}
Here $\GW(k)$ denotes the Grothendieck-Witt group of finite rank non-degenerate symmetric bilinear forms (see for example \cite[Chapter II]{MR2104929} for the definition of $\GW(k)$) which is generated by $\langle\alpha\rangle$ for $\alpha\in k^{\times}$ where $\langle \alpha\rangle$ is the class of the form $k\times k\rightarrow k$, $(x,y)\mapsto \alpha xy$ in $\GW(k)$.

Let $Y$ be a smooth and proper scheme over a field $k$ and $E\rightarrow Y$ a relatively oriented vector bundle of rank equal to the dimension of $Y$. Then a general section $\sigma:Y\rightarrow E$ of the bundle has finitely many zeros. Kass and Wickelgren define the $\A^1$-Euler number $e^{\A^1}(E)$ to be the sum of local indices $\ind_x\sigma$ at the zeros of $\sigma$ where the local index $\ind_x\sigma$ is the local $\A^1$-degree at the zero $x$ in local coordinates and a trivialization of $E$ compatible with the relative orientation of $E$ in the sense of \cite[Definition 21]{MR4247570}.

Let $\mathcal{S}\rightarrow\Gr(2,5)$ be the tautological bundle over the Grassmannian of lines in $\PP^4$. We show in Proposition \ref{prop: relative orienatation} that the vector bundle \[\mathcal{E}:=\Sym^5\mathcal{S}^*\rightarrow\Gr(2,5)\] is relatively orientable. Furthermore, $\dim \Gr(2,4)=6=\operatorname{rank}\mathcal{E}$. So the $\A^1$-Euler number of $\mathcal{E}$ is well-defined. Let $X=\{f=0\}\subset
\PP^4$ be a general quintic threefold. Then $f$ defines a general section $\sigma_f$ of $\mathcal{E}$ by restricting the defining polynomial $f$ to the lines in $\PP^4$. The lines on $X$ are the lines with $f\vert_l=0$, i.e., the zeros of the section $\sigma_f$. Whence, the $\A^1$-Euler number of $\mathcal{E}$ is by definition the sum of local indices at the lines on a general quintic threefold. Our main result in the first half of this paper is the following.


\begin{theorem}
\label{thm: main thm 1}
Let $X=\{f=0\}\subset\PP^4$ be a general quintic threefold and let $l\subset X$ be a $k$-line on $X$.
The type of $l$ is equal to the local index at the corresponding zero of the section $\sigma_f:\Gr(2,5)\rightarrow \mathcal{E}$.
\end{theorem}
It follows that the $\A^1$-Euler number of $\mathcal{E}$ is equal to the sum of types of lines on a general quintic threefold $X$
\[e^{\A^1}(\mathcal{E})=\sum_{l\subset X}\ind_l\sigma_f=\sum_{l\subset X}\Tr_{k(l)/k}\operatorname{Type}(l)\in \GW(k).\]
Here, $\Tr_{k(l)/k}$ denotes the trace form, that is the composition of a finite rank non-degenerate symmetric bilinear form with the field trace. 

Having defined the type, we want to compute the $\A^1$-Euler number $e^{\A^1}(\mathcal{E})$, that is, sum up all the types of lines on a general quintic threefold. However, many `nice' quintic threefolds are not general and define sections with non-isolated zeros. Yet, a general deformation of such a threefold contains only finitely many lines.
We define the \emph{dynamic Euler number} of a relatively oriented vector bundle $\pi:E\rightarrow Y$ over a smooth and proper scheme $Y$ over $k$ with $\operatorname{rank}E=\dim Y$, to be the sum of the local indices at the zeros of a deformation $\sigma_t$ of a section $\sigma$ valued in $\GW(k((t)))$, that is, the $\A^1$-Euler number of the base changed bundle $(E\rightarrow Y)_{k((t))}$ expressed as the sum of local indices at the zeros of $\sigma_t$. By Springer's Theorem (Theorem \ref{thm: Springer}) the dynamic Euler number completely determines the $\A^1$-Euler number $e^{\A^1}(E)\in \GW(k)$.

As an application, we compute the dynamic Euler number of $\mathcal{E}=\Sym^5\mathcal{S}^*\rightarrow\Gr(2,5)$ with respect to the section defined by the Fermat quintic threefold \[X=\{X_0^5+X_1^5+X_2^5+X_3^5+X_4^5=0\}\subset\PP^4.\]
There are infinitely many lines on $X$ and the section $\sigma_F$ defined by the Fermat does not have any isolated zeros.
For a general deformation $X_t=\{F_t=F+tG+t^2H+\dots=0\}$ of the $X$ Albano and Katz find 2875 distinguished complex lines on $X$ which are the limits of lines the deformation \cite{MR1024767}. Their computation still works over a field $k$ of characteristic not equal to 2 or 5 and adding up the local indices at the deformed lines $l_t$ on the deformation $X_t$, we  get  
\[e^{\text{dynamic}}(\mathcal{E})=1445\langle1\rangle+1430\langle-1\rangle\in \GW(k((t))).\]
The $\A^1$-Euler number $e^{\A^1}(\mathcal{E})$ is the unique element in $\GW(k)$ that is mapped to $e^{\operatorname{dynamic}}(\mathcal{E})$ by $i:\GW(k)\rightarrow \GW(k((t)))$ defined by $\langle a\rangle\mapsto\langle a\rangle$, that is,
\[e^{\A^1}(\mathcal{E})=1445\langle1\rangle+1430\langle-1\rangle\in\GW(k).\]
Combining this result with Theorem \ref{thm: main thm 1} we get that
\begin{equation}
e^{\A^1}(\mathcal{E})=\sum\operatorname{Tr}_{k(l)/k}\text{Type}(l)=1445\langle1\rangle+1430\langle-1\rangle\in \GW(k)
\end{equation}
where the sum runs over the lines on a general quintic threefold and $k$ is a field of characteristic not equal to 2 or 5.

\begin{remark}
Levine has already computed the $\A^1$-Euler number of $\mathcal{E}=\Sym^5\mathcal{S}^*\rightarrow\Gr(2,5)$ to be \[e^{\A^1}(\mathcal{E})=1445\langle1\rangle+1430\langle-1\rangle\in \GW(k)\] 
in \cite[Example 8.2]{LEVINE_2019} using the theory of Witt-valued characteristic classes. 
\end{remark}
Our computation reproves this result without using this theory and gives a new technique to compute `dynamic' characteristic classes in the motivic setting. Additionally, we get a refinement of his result. 
The local index is only defined for isolated zeros. However, we can define the `local index' $\ind_l\sigma_{F_t}$ at a line on the Fermat that deforms with a general deformation defined by $F_t$ to lines $l_{t,1},\dots,l_{t,m}$, to be the unique element in $\GW(k)$ that is mapped to $\sum_{i=1}^m\ind_{l_{t,i}}\sigma_{F_t}\in\GW(k((t)))$ by $i:\GW(k)\mapsto\GW(k((t)))$ where $\sigma_{F_t}$ is the section defined by restricting $F_t$. 
\begin{theorem}
Assume $\operatorname{char}k\neq 2,5$.
There are well-defined local indices $\ind_l\sigma_{F_t}\in \GW(k)$ of the lines on the Fermat quintic threefold that deform with a general deformation $X_t$ of the Fermat depending on the deformation and
\[e^{\A^1}(\mathcal{E})=\sum_{l\text{ deforms with }F_t}\ind_l\sigma_{F_t}=1445\langle1\rangle+1430\langle-1\rangle\in \GW(k).\]
\end{theorem}
Note that $i:\GW(k)\rightarrow\GW(k((t)))$ is not an isomorphism, it is injective but not surjective. So it is not clear, that there exist such local indices in $\GW(k)$. This indicates that there should be a more general notion of the local index which can also be defined for non-isolated zeros of a section which deform with a general deformation. 

Recall from \cite[Chapter 6]{MR1644323} that classically the intersection product splits up as a sum of cycles supported on the \emph{distinguished varieties}. We observe that this is true for the intersection product of $\sigma_F$ by the zero section in the enriched setting.
\begin{theorem}
The sum of local indices at the lines on a distinguished variety of the intersection product of $\sigma_F$ by $s_0$ that deform with a general deformation is independent of the chosen deformation. So for a distinguished variety $Z$ of this intersection product, there is a well-defined local index $\ind_Z\sigma_F\in \GW(k)$ and
\[e^{\A^1}(\mathcal{E})=\sum \ind_Z\sigma_F=1445\langle1\rangle+1430\langle-1\rangle\in\GW(k)\]
where the sum runs over the distinguished varieties.
\end{theorem}
This is the first example of a quadratic dynamic intersection. 

Kass and Wickelgren introduced the $\A^1$-Euler number in \cite{MR4247570} to count lines on a smooth cubic surface. Since then, their work has been used to compute several other $\A^1$-Euler numbers valued in $\GW(k)$ \cite{MR4253146}, \cite{MR4211099}, \cite{pauli2020computing}, \cite{MR4237952}. The $\A^1$-Euler number fits into the growing field of $\A^1$-enumerative geometry. Other related results are \cite{BW},\cite{MR4071212}, \cite{MR4162156}, \cite{MR3909901}, \cite{kobin2019mathbba1local}, \cite{MR4198841},\cite{LEVINE_2019},\cite{MR4071217}.

\section{The local index}
\label{section: local index}

Let $X=\{f=0\}\subset \PP^{4}$ be a degree $5$ hypersurface. Then $f$ defines a section $\sigma_f$ of the bundle $\mathcal{E}:=\operatorname{Sym}^5\mathcal{S}^*\rightarrow \operatorname{Gr}(2,5)$ by restricting the polynomial $f$ to the lines in $\PP^4$. Here, $\mathcal{S}\rightarrow \operatorname{Gr}(2,5)$ denotes the tautological bundle on the Grassmannian $\operatorname{Gr}(2,5)$ of lines in $\PP^4$.  A line $l$ lies on $X$ if and only if $f\vert_l=0$ which occurs if and only if $\sigma_f(l)=0$.
\begin{definition} Assume $x$ is an isolated zero of $\sigma_f$. The \emph{local index} $\operatorname{ind}_x\sigma_f$ of $\sigma_f$ at $x$ is the local $\A^1$-degree (see \cite{MR3909901}) in coordinates and a trivialization of $\mathcal{E}$ compatible with a fixed relative orientation of $\mathcal{E}$ at $x$. 
\end{definition}
The $\A^1$-Euler number of $\mathcal{E}=\Sym^5\mathcal{S}^*\rightarrow \Gr(2,5)$ is by definition \cite{MR4247570} the sum of local indices at the zeros of a section $\sigma_f$ with only isolated zeros
\[e^{\A^1}(\mathcal{E}):=\sum_{x\text{ a zero of }\sigma_f}\operatorname{ind}_x\sigma_f\in \GW(k).\] 
In other words, the $\A^1$-Euler number of $\mathcal{E}$ is a `quadratic count' of the lines on quintic threefold $X$ with finitely many lines.

\begin{remark}
The $\A^1$-Euler number is independent of the chosen section (given it has only isolated zeros) by \cite[Theorem 1.1]{BW}. 
\end{remark}

In this section we define a relative orientation of $\mathcal{E}$ and coordinates and local trivializations of $\mathcal{E}$ compatible with it. Then we give an algebraic description of the local index at a zero of a section defined by a general quintic threefold.

\subsection{Relative orientability}

We recall the definition of a relative orientation from \cite[Definition 17]{MR4247570}.
\begin{definition}
A vector bundle $\pi:E\rightarrow X$ is \emph{relatively orientable}, if there is a line bundle $L\rightarrow X$ and an isomorphism $\phi:\operatorname{Hom}(\det TX,\det E)\xrightarrow{\cong}L^{\otimes 2}$. Here, $TX\rightarrow X$ denotes the tangent bundle on $X$. We call $\phi$ a \emph{relative orientation} of $E$.
\end{definition}

\begin{remark}
When $TX$ and $E$ are both orientable (that is, both are isomorphic to a square of a line bundle), then $E$ is relatively orientable. However, the example of lines on a quintic threefold shows, that there are relatively orientable bundles which are not orientable.
\end{remark}

\begin{proposition}
\label{prop: relative orienatation}
Denote by $G$ the Grassmannian $\Gr(2,5)$.
The vector bundle $\mathcal{E}\rightarrow G$ is relatively orientable. More precisely, there is a canonical isomorphism $\phi:\operatorname{Hom}(\det TG,\det \mathcal{E})\xrightarrow{\cong}L^{\otimes 2}$ with $L=\det\mathcal{Q}^*\otimes (\det\mathcal{S}^*)^{\otimes 6}$ where $\mathcal{Q}$ denotes the quotient bundle on the Grassmannian $G$.
\end{proposition}
\begin{proof}
It follows from the natural isomorphism $TG\cong \mathcal{S}^*\otimes \mathcal{Q}$ that there is a canonical isomorphism $\det TG\cong (\det \mathcal{S}^*)^{\otimes 3}\otimes (\det \mathcal{Q})^{\otimes 2}$ \cite[Lemma 13]{MR3176620}. The determinant of $\mathcal{E}$ is canonically isomorphic to $(\det \mathcal{S}^*)^{\otimes 15}$. Thus, 
\begin{align*}
\operatorname{Hom}(\det TG,\det \mathcal{E)}&\cong (\det TG)^*\otimes \det \mathcal{E}\\
&\cong (\det \mathcal{S})^{\otimes 3}\otimes (\det \mathcal{Q}^*)^{\otimes 2}\otimes(\det \mathcal{S}^*)^{\otimes 15}\\
 &\cong (\det \mathcal{Q}^*)^{\otimes 2}\otimes(\det \mathcal{S}^*)^{\otimes 12}.
\end{align*}
\end{proof}
\subsection{Local coordinates}
\label{subsection: local coordinates}
As in \cite[Definition 4]{MR4247570} we define the \emph{field of definition} of a line $l$ to be the residue field of the corresponding closed point in $\Gr(2,5)$. We say that a line on a quintic threefold $X=\{f=0\}\subset\PP^4$ is \emph{simple} and \emph{isolated} if the corresponding zero of the section $\sigma_f$ is simple and isolated.
\begin{lemma}
\label{lemma: separable field ext}
A line on a general quintic threefold $X=\{f=0\}\subset \PP^4$ is simple and isolated and its field of definition is a separable field extension of $k$.
\end{lemma}
\begin{proof}
By \cite[Theorem 6.34]{MR3617981} the Fano scheme of lines on $X$ is geometrically reduced and zero dimensional. It follows that the zero locus $\{\sigma_f=0\}$ is zero dimensional and geometrically reduced. It particular, it consists of finitely many, simple zeros. A line $l$ on $X$ with field of definition $L$ defines a map $\Spec L\rightarrow \{\sigma_f=0\}$ and $\Spec L$ is a connected component of $\{\sigma_f=0\}$. In particular, $\Spec L$ is geometrically reduced, which implies that $l$ is simple and the field extension $L/k$ is separable.
\end{proof}
Let $L$ be the field of definition of a line $l$ on a general quintic threefold $X=\{f=0\}\subset\PP^4$ and let $(\sigma_f)_{L}$ be the base changed section of the base changed bundle $\mathcal{E}_{L}$. By Lemma \ref{lemma: separable field ext} $L/k$ is separable. 
Recall that the \emph{trace form} $\Tr_{L/k}(\beta)$ of a finite rank non-degenerate symmetric bilinear form $\beta$ over a finite separable field extension $L$ over $k$ is the finite rank non-degenerate symmetric bilinear form over $k$ defined by the composition
\[V\times V\xrightarrow{\beta}L\xrightarrow{\Tr_{L/k}}k\]
where $\Tr_{L/k}$ denotes the field trace.
Denote by $(\sigma_f)_L$ and $\Gr(2,5)_L$ be the base change to $L$ and let $\Spec L\xrightarrow{l_L} \Gr(2,5)_L$ be the base change of $\Spec k\xrightarrow{l}\Gr(2,5)$.  By \cite[Proposition 34]{MR4247570} $\ind_l\sigma_f=\Tr_{L/k}\ind_{l_{L}}(\sigma_{f})_{L}$. Therefore, to compute the local index at $l$, we can assume that $l\subset X$ is $k$-rational and take the trace form if necessary.
\begin{remark}
\label{remark: etale algebra}
One can define traces in a much more general setting. Let $A$ be a commutative ring with $1$ and assume $B$ is a finite projective $A$-algebra. Then the trace $\Tr_{B/A}:B\rightarrow A$ is the map that sends $b\in B$ to the trace of the multiplication map $m_b(x)=b\cdot x$ (just like in the field case). The trace is transitive in the following sense: If $C$ is a finite projective $B$-algebra, then $\Tr_{C/A}=\Tr_{B/A}\circ\Tr_{C/B}$. In case $B$ is \'{e}tale/separable over $A$, we have a map $\Tr_{B/A}:\GW(B)\rightarrow \GW(A)$ defined exactly as in the case of separable field extensions. A special case is that $B$ be a finite  \'{e}tale $k$-algebra where $k$ is a field. Then $B$ is isomorphic to the product $L_1\times\ldots\times L_s$ of finitely many separable finite field extensions $L_1,\ldots,L_s$ of $k$ and $\Tr_{B/k}=\sum_{i=1}^s\Tr_{L_i/k}$ and the trace form $\Tr_{B/k}(\beta)$ equals the sum of trace forms $\Tr_{B/k}(\beta)=\sum_{i=1}^s\Tr_{L_i/k}(\beta_{L_i})$ where $\beta_{L_i}$ is the restriction of $\beta$ to $L_i$. 
\end{remark}
After a coordinate change, assume that $l=(0:0:0:u:v)$, i.e., $l=\{x_1=x_2=x_3=0\}\subset\PP(k[x_1,x_2,x_3,u,v])$. We define coordinates on the Grassmannian around $l$.
Let $e_1,e_2,e_3,e_4,e_5$ be the standard basis of $k^5$. The line $l$ is the 2-plane in $k^5$ spanned by $e_4$ and $e_5$. Let $U:=\A^6=\Spec k[x,x',y,y',z,z']\subset \Gr(2,5)$ be the open affine subset of lines spanned by $xe_1+ye_2+ze_3+e_4$ and $x'e_1+y'e_2+z'e_3+e_5$. Note that the line $l$ is the origin in $U$.

\begin{remark}
In general, a $d$-dimensional smooth scheme $X$ does not have a covering by affine spaces and one has to use \emph{Nisnevich coordinates} around a closed point $x$, that is an \'{e}tale map $\psi:U\rightarrow \A^d$ for a Zariski neighborhood $U$ of a closed point $x$ such that $\psi$ induces an isomorphism on the residue field $k(x)$.
Since $\psi$ is \'{e}tale, it defines a trivialization of $TX\vert_U$ where $TX\rightarrow X$ is the tangent bundle.
Clearly, $U=\A^6$ are Nisnevich coordinates around $l$.
\end{remark}

Recall from \cite[Definition 21]{MR4247570} that a trivialization of $\mathcal{E}\vert_U$ is compatible with the relative orientation $\phi:\operatorname{Hom}(\det TG,\det \mathcal{E})\xrightarrow{\cong}L^{\otimes 2}$ with $L=\wedge^3 \mathcal{Q}^*\otimes (\wedge^2\mathcal{S}^*)^{\otimes 6}$ if the distinguished element of $\operatorname{Hom}(\det TG\vert_U,\det \mathcal{E}\vert_U)$ sending the distinguished basis of $\det TX\vert_U$ to the distinguished basis of $\det\mathcal{E}\vert_U$, is sent to a square by $\phi$.

We choose a trivialization of $\mathcal{E}\vert_U$ and show that it is compatible with our coordinates and relative orientation similarly as done in \cite[Proposition 45, Lemma 46, Proposition 47]{MR4247570} but we avoid change of basis matrices.
Let $\tilde{e}_1:=e_1$, $\tilde{e}_2:=e_2$, $\tilde{e}_3:=e_3$, $\tilde{e}_4:=xe_1+ye_2+ze_3+e_4$ and $\tilde{e}_5:=x'e_1+y'e_2+z'e_3+e_5$. Then $(\tilde{e}_1,\dots,\tilde{e}_5)$ is a basis of $(k[x,x',y,y',z,z']^5)$. Let $(\tilde{\phi}_1,\dots,\tilde{\phi}_5)$ be its dual basis.
\begin{lemma}
\label{lemma: trivialization of E}
The restrictions $TG\vert_{U}$ and $\mathcal{E}\vert_U$ have bases given by 
\begin{equation}
\label{eq:basis TGr}
\tilde{\phi_4}\otimes \tilde{e}_1,\tilde{\phi_5}\otimes\tilde{e}_1,\tilde{\phi_4}\otimes \tilde{e}_2,\tilde{\phi_5}\otimes\tilde{e}_2,\tilde{\phi_4}\otimes \tilde{e}_3,\tilde{\phi_5}\otimes\tilde{e}_3
\end{equation}
and 
\begin{equation}
\label{eq:basis Sym5}
\tilde{\phi}_4^5,\tilde{\phi}_4^4\tilde{\phi}_5,\tilde{\phi}_4^3\tilde{\phi}_5^2,\tilde{\phi}_4^2\tilde{\phi}_5^3,\tilde{\phi}_4\tilde{\phi}_5^4,\tilde{\phi}_5^5
\end{equation}
respectively.
The bases define trivializations of $TG\vert_U$ and $\mathcal{E}\vert_U$ and a distinguished basis element of \\$\operatorname{Hom}(\det TG\vert_U,\det\mathcal{E}\vert_U)$, namely the morphism that sends the wedge product of \eqref{eq:basis TGr} to the wedge product of \eqref{eq:basis Sym5}. The image of this distinguished basis element under $\phi$ defined in Proposition \ref{prop: relative orienatation} is 
\[(\tilde{e}_1\wedge \tilde{e}_2\wedge \tilde{e}_3)^{\otimes -2}\otimes (\tilde{\phi}_4\wedge\tilde{\phi}_5)^{\otimes 12}\]
which is a square. In particular, the chosen trivialization of $\mathcal{E}\vert_U$ is compatible with the relative orientation $\phi$.
\end{lemma}
\begin{proof}

The canonical isomorphism $\det TG\xrightarrow{\cong}(\det \mathcal{S}^*)^{\otimes 3}\otimes (\det \mathcal{Q})^{\otimes 2}$ in Lemma \ref{prop: relative orienatation} sends \[\tilde{\phi_4}\otimes \tilde{e}_1\wedge\tilde{\phi_5}\otimes\tilde{e}_1\wedge\tilde{\phi_4}\otimes \tilde{e}_2\wedge\tilde{\phi_5}\otimes\tilde{e}_2\wedge\tilde{\phi_4}\otimes \tilde{e}_3\wedge\tilde{\phi_5}\otimes\tilde{e}_3\mapsto(\tilde{\phi}_4\wedge\tilde{\phi}_5)^3\otimes (\tilde{e}_1\wedge\tilde{e}_2\wedge\tilde{e}_3)^2\]
and the canonical isomorphism $\det \Sym^5\mathcal{S}^*\xrightarrow{\cong }(\det \mathcal{S}^*)^{15}$ sends 
\[\tilde{\phi}_4^5\wedge\tilde{\phi}_4^4\tilde{\phi}_5\wedge\tilde{\phi}_4^3\tilde{\phi}_5^2\wedge\tilde{\phi}_4^2\tilde{\phi}_5^3\wedge\tilde{\phi}_4\tilde{\phi}_5^4\wedge\tilde{\phi}_5^5\mapsto(\tilde{\phi}_4\wedge\tilde{\phi}_5)^{5+4+3+2+1}.\]
It follows that $\phi$ sends the distinguished basis element to 
$(\tilde{\phi}_4\wedge\tilde{\phi}_5)^{15-3}\otimes (\tilde{e}_1\wedge\tilde{e}_2\wedge\tilde{e}_3)^{-2}$.

\end{proof}


\subsection{The local index at an isolated, simple zero of $\sigma_f$}
\label{subsection: the local index an an isolated zero}
Let $X=\{f=0\}\subset \PP^4$ be a quintic threefold and let $l\subset X$ be an isolated, simple line defined over $k$. Let $u,v,x_1,x_2,x_3$ be the coordinates on $\PP^4$. Again we assume that $l=\{x_1=x_2=x_3=0\}\subset \PP^4$. Then the definining polynomial of $X$ is of the form $f=x_1P_1(u,v)+x_2P_2(u,v)+x_3P_3(u,v)+Q(u,v,x_1,x_2,x_3)$ where $Q\in (x_1,x_2,x_3)^2$, that is, $Q$ vanishes to degree at least two on $l$, and 
\begin{align*}
P_1(u,v)&=a_{4,1}u^4+a_{3,1}u^3v+a_{2,1}u^2v^2+a_{1,1}uv^3+a_{0,1}v^4,\\
P_2(u,v)&=a_{4,2}u^4+a_{3,2}u^3v+a_{2,2}u^2v^2+a_{1,2}uv^3+a_{0,2}v^4,\\
P_3(u,v)&=a_{4,3}u^4+a_{3,3}u^3v+a_{2,3}u^2v^2+a_{1,3}uv^3+a_{0,3}v^4
\end{align*}
are homogeneous degree 4 polynomials in $u$ and $v$. 
\begin{proposition}
\label{prop: local index algebraically}
The local index of $\sigma_f$ at $l$ is equal to $\langle\det A_{P_1,P_2,P_3}\rangle\in \operatorname{GW}(k)$ with
\begin{equation}
\label{eq: matrix A}
A_{P_1,P_2,P_3}=\begin{pmatrix}
a_{4,1} & 0 & a_{4,2}& 0& a_{4,3}& 0\\
a_{3,1} & a_{4,1} & a_{3,2}& a_{4,2}& a_{3,3}& a_{4,3}\\
a_{2,1} & a_{3,1} & a_{2,2}& a_{3,2}& a_{2,3}& a_{3,3}\\
a_{1,1} & a_{2,1} & a_{1,2}& a_{2,2}& a_{1,3}& a_{2,3}\\
a_{0,1} & a_{1,1} & a_{0,2}& a_{1,2}& a_{0,3}& a_{1,3}\\
0 & a_{0,1} & 0& a_{0,2}& 0& a_{0,3}
\end{pmatrix}
\end{equation}
where the entries of the matrix $A$ are coefficients of $P_1$, $P_2$ and $P_3$.
\end{proposition}
\begin{proof}
Note that $l$ is 0 in the coordinates from subsection \ref{subsection: local coordinates}.
So the local index at $l$ is the local $\A^1$-degree at 0 in the chosen coordinates and trivialization of $\mathcal{E}$. 
Since $l$ is simple, that is, 0 is a simple zero in the chosen coordinates and trivialization, it follows from \cite{MR3909901} that the local $\A^1$-degree at 0 is equal to the determinant of the jacobian matrix of $\sigma_f$ in the local coordinates and trivialization defined in Lemma \ref{lemma: trivialization of E} evaluated at 0.
The section $\sigma_f$ in those local coordinates and trivialization is the morphism $(f_1,\dots,f_6):\A^6\rightarrow \A^6$ defined by the 6 polynomials which are the coefficients of $u^5,u^4v,\dots,v^5$ in 
\begin{align*}
&f(xu+x'v,yu+y'v,zu+z'v,u,v)\\
&=(xu+x'v)P_1(u,v)+(yu+y'v)P_2(u,v)+(zu+z'v)P_3(u,v)+Q(xu+x'v,yu+y'v,zu+z'v,u,v).
\end{align*}
Since $Q$ vanishes to at least order two on the line, the partial derivative of $Q$ in any of the six directions $x,x',y,y',z,z'$ evaluated at 0 vanishes. So the partial derivative $\frac{f(xu+x'v,yu+y'v,zu+z'v,u,v)}{\partial x}\vert_0$ in $x$-direction evaluated at 0 is $uP_1(u,v)$ and the coeffiecients of $u^5,u^4v,\dots,v^5$ in $uP_1(u,v)$ are $a_{4,1},a_{3,1},a_{2,1},a_{1,1},a_{0,1},0$, that is, the first column of \eqref{eq: matrix A}. Similarly, one computes the remaining columns of \eqref{eq: matrix A}.
\end{proof}
\begin{remark}
\label{remark: simple line}
We have seen in Lemma \ref{lemma: separable field ext} that on a general quintic threefold, all lines are isolated and simple. An isolated line $l$ is not simple if and only if the derivative of $\sigma_f$ at $l$ vanishes if and only if $\det A_{P_1,P_2,P_3}=0$. In the case that $l$ is isolated but not simple, one can use the \emph{EKL-form} \cite{MR3909901} to compute the corresponding local index. 
\end{remark}
\begin{lemma}
\label{lemma: detA=0}
We have $\det A_{P_1,P_2,P_3}=0$ if and only if there are degree 1 homogeneous polynomials $r_1,r_2,r_3\in k[u,v]$, not all zero, such that $r_1P_1+r_2P_2+r_3P_3=0$.
\end{lemma}
\begin{proof}
This is Lemma 3.2.2 (3) in \cite{MR4230388}. 
\end{proof}
\section{Definition of the type}
\label{section: defn of type}
In this section we provide the definition of the \emph{type} of an isolated, simple line.
For the definition of the type of a line, we need to work over a field of characteristic not equal to 2 because there are involutions involved.
Again we assume by base change that $k(l)=k$ and that $l=\{x_1=x_2=x_3=0\}\subset \{f=0\}=X\subset\PP(k[x_1,x_2,x_3,u,v])$.
Recall from section \ref{subsection: the local index an an isolated zero} that under these assumptions the polynomial $f$ defining the quintic threefold $X$ is of the form $f=x_1P_1+x_2P_2+x_3P_3+Q$ where $Q\in (x_1,x_2,x_3)^2$ and $P_1,P_2,P_3\in k[u,v]_4$.

Let $C:l\rightarrow \PP^2$ be the degree 4 rational plane curve $(u:v)\mapsto (P_1(u,v):P_2(u,v):P_3(u,v))$. Then $C$ has the following geometric description: The 3-planes in $\PP^4$ containing $l$ can be parametrized by a $\PP^2$.
Identify the tangent space $T_pX$ at a point $p\in l$ in $X$ with a 3-plane in $\PP^4$. Then the corresponding 4-plane in $k^5$ has normal vector $(P_1(p),P_2(p),P_3(p),0,0)$. Therefore, $C$ maps a point $p\in l$ to its tangent space $T_pX$  in $\PP^4$, i.e., $C$ is the \emph{Gau\ss} \emph{map}. By the Castelnuovo count (see e.g. \cite{MR770932}),
a general degree 4 rational plane curve has 3 double points.
Furthermore, Finashin and Kharlamov show that the number of double points on $C$, given it is finite, is always 3 (possibly counted with multiplicities) (see \cite[Proposition 4.3.3]{MR4230388}). However, there could also be infinitely many double points on $C$. We will deal with this case in \ref{subsubsection: infinitely many double points}. For now we assume that $C$ has 3 double points. That means that there are 3 pairs of points $(r_i,s_i)$ on the line $l$ which have the same tangent space in $X$, i.e., $T_{r_i}X=T_{s_i}X$ for $i=1,2,3$. 

Let $M$ be one double point of the curve $C$ with field of definition (= the residue field of $M$ in $\PP^2$) $L_M$ and let $D$ be the corresponding degree 2 divisor on $l$. Let $H_t$ be the pencil of lines in $\PP^2$ through $M$. Then $D_t=H_t\cap C=D+D_t^r$ defines a pencil of degree $4$ divisors on $l$. 
\begin{lemma}
\label{lemma: base point free}
When $l$ is a simple line, the residual pencil $D_t^r$ of degree 2 divisors on $l$ is base point free, that means, there is no point on $l$ where every element of the pencil $D_t^r$ vanishes.
\end{lemma}
\begin{proof}

This follows from \cite[Lemma 4.1.1 and $\S5.1$]{MR4230388}. However, we reprove the result in our notation and setting.
Without loss of generality, we can assume that $M=(1:0:0)$. That means, there are two points on the line $l$ that are sent to $(1:0:0)$ by $C$. Let $Q$ be the homogeneous degree 2 polynomial in $u$ and $v$ that vanishes on those two points. It follows that $Q$ divides both $P_2$ and $P_3$.
Let $Q_1^M:=\frac{P_2}{Q}$ and $Q_2^M:=\frac{P_3}{Q}$ be the two degree 2 homogeneous polynomials in $u$ and $v$. The support of $D_t^r$ is $\{t_0Q_1^M+t_1Q_2^M=0\}\subset l\times \PP^1\cong \PP^1\times \PP^1$. If $D_t^r$ had a basepoint, then $Q_1^M$ and $Q_2^M$ would have a common factor and thus $P_2$ and $P_3$ had a common degree 3 factor and there would be nonzero degree 1 homogeneous polynomials $r_2$ and $r_3$ in $u$ and $v$ such that $r_2P_2+r_3P_3=0\cdot P_1+r_2P_2+r_3P_3=0$. But then $\det A_{P_1,P_2,P_3}=0$ by Lemma \ref{lemma: detA=0} where $A_{P_1,P_2,P_3}$ is the matrix from Proposition \ref{prop: local index algebraically}, that is, $l$ is not a simple line.
\end{proof}
It follows that $D_t^r$ defines a double covering $(Q_1^M:Q_2^M):\PP^1_{L_M}\rightarrow \PP^1_{L_M}$. 
\begin{definition}
\label{def:Segre involution and fixed points}
We call the nontrivial element of the Galois group of the double covering $(Q_1^M,Q_2^M):\PP^1_{L_M}\rightarrow \PP^1_{L_M}$ a \emph{Segre involution} and denote it by
\[i_M:\PP^1_{L_M}\rightarrow\PP^1_{L_M}.\]
Its fixed points are called \emph{Segre fixed points}.
\end{definition}
Figure \ref{img:grafik-dummy} illustrates what the Segre involution $i_M$ does. Each element $H_{t_0}$ of the pencil $H_t$ of lines through $M$ intersects the curve $C$ in two additional points $C(p)$ and $C(q)$. The involution $i_M$ swaps the preimages $p$ and $q$ of those two points on $l$.

Geometrically, the Segre involution can be described as follows.
Let $T=T_s(X\otimes L_M)=T_r(X\otimes L_M)$ where $r,s\in \PP^1$ with $C(s)=C(r)=M$. Then to a point $p\in l\otimes L_M$ there exists exactly one $q\in l\otimes L_M$ such that $T\cap T_p(X\otimes L_M)=T\cap (T_qX\otimes L_M)$. The involution $i_M$ swaps $p$ and $q$. 
\begin{figure}
	\centering
	\includegraphics[width=0.9\columnwidth]{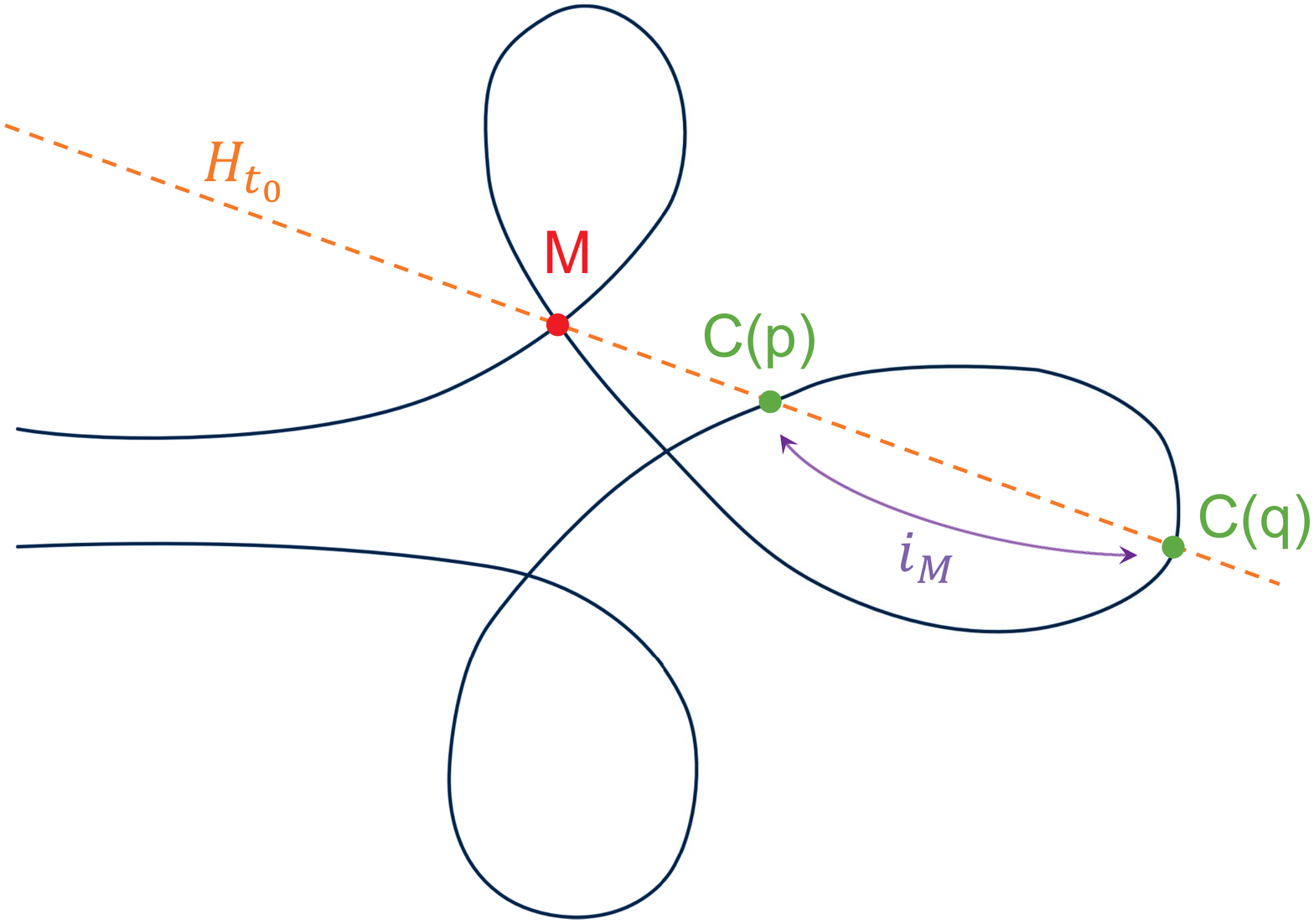}
	\caption{involution $i_M$}
	\label{img:grafik-dummy}
\end{figure}

The Segre fixed points of $i_M$ are defined over $L_M(\sqrt{\alpha_M})$ for $\alpha_M\in L_M^{\times}/(L_M^{\times})^2$. We say $\alpha_M$ is the \emph{degree} of the Segre involution $i_M$.
\begin{remark}
The involution $i_M$ is a self map of the motivic sphere $\PP^1_{L_M}$ and thus, we can assign an $\A^1$-degree $\deg^{\A^1}(i_M)$ to $i_M$ valued in $\GW(L_M)$ \cite{MR2934577}. It follows from \cite[Corollary 13]{MR4247570} that $\langle-1\rangle\deg^{\A^1}(i_M)=\langle\alpha_M\rangle$ in $\GW(L_M)$.
\end{remark}
\begin{lemma}
\label{lemma: degree i=Res}
The degree of the Segre involution $i_M:\PP^1_{L_M}\rightarrow\PP^1_{L_M}$ is equal to the resultant $\Res(Q_1^M,Q_2^M)$ in $L_M^{\times}/(L_M^{\times})^2$.
\end{lemma}
\begin{proof}
This follows from the proof of \cite[Proposition 14]{MR4247570}.
\end{proof}

We want to define the type of $l$ to be the product of the three Segre involutions. We need a well-defined element of $k^{\times}/(k^{\times})^2$ in order to get a well-defined element of $\GW(k)$. If all three double points are defined over $k$, the degrees of the Segre involutions are elements of $k^{\times}/(k^{\times})^2$, so the product is a well-defined element of $k^{\times }/(k^{\times })^2$. 
If a double point is defined over a proper field extension $L$ over $k$, its degree is an element of $L^{\times}/(L^{\times})^2$ and its contribution to the product of the degrees of the three involutions is only well-defined up to a square in $L^{\times}$ which might not be a square in $k^{\times}$. However, we will see that the double points and their degrees come in Galois orbits. The product of Galois conjugate degrees is equal to the norm and lies in $k^{\times}$. In particular, this is well-defined up to a square in $k^{\times}$.

\begin{definition}
\label{defn: type}
Let $\M$ be the locus of double points of $C$ and assume $\operatorname{char}k\neq 2$.
The \emph{type} of the (isolated and simple) $k$-line $l$ on $X$ is 
\begin{equation}
\label{eq: defn of type}
\operatorname{Type}(l):=\langle\prod (N_{L_M/k}(\alpha_M))\rangle \in \GW(k)
\end{equation}
where the product runs over the Galois orbits of $\M$. 
\end{definition}
\begin{remark}
Recall that a line on a smooth cubic surface gives rise to a (single) Segre involution in a similar way: To each point on the line there is exactly one other point with the same tangent space and the Segre involution swaps those two points. In \cite{MR4247570} Kass and Wickelgren define the \emph{type} of a line on a smooth cubic surface to be the degree of this Segre involution. 
\end{remark}

\begin{remark}
If all three double points are pairwise different, the fixed points of the involutions $i_M$ form a degree 6 divisor on the line $l$ which corresponds to an \'{e}tale $k$-algebra $A$. Then the type of $l$ is equal to the discriminant of $A$ over $k$
\[\operatorname{Type}(l)=\langle\operatorname{disc}(A/k)\rangle\in \GW(k).\]
Recall that the type of a line on a smooth cubic surface is also the discriminant of the fixed point scheme of a Segre involution \cite[Corollary 13]{MR4247570}.
\end{remark}

\section{The type is equal to the local index}

\begin{theorem}
\label{thm: main thm}
Let $X=\{f=0\}\subset \PP^4$ be a quintic threefold and $l\subset X$ be a simple and isolated line with field of definition a field $k$ of characteristic not equal to 2.
The type of $l$ is equal to the local index at $l$ of $\sigma_f$ in $\GW(k)$.
\end{theorem}
\begin{corollary}
For $X$ a general quintic threefold we have

\begin{equation}
e^{\A^1}(\mathcal{E})=\sum_{l\subset X}\ind_l\sigma_f=\sum_{l\subset X}\Tr_{k(l)/k}\ind_{l_{k(l)}}(\sigma_f)_{k(l)}=\sum_{l\subset X}\Tr_{k(l)/k}\operatorname{Type}(l)\in\GW(k)
\end{equation}

when $k$ is a field with $\operatorname{char}k\neq 2$.
\end{corollary}
This means that when we count lines on a general quintic threefold weighted by the product of the degrees of the fixed points of the Segre involutions corresponding to the three double points of the Gauss map of the line,  we get an invariant element of $\GW(k)$.
\subsection{An algebraic interpretation of the product of resultants}
We have seen that the type of a line is equal to the product of resultants of quadratic polynomials in Lemma \ref{lemma: degree i=Res}. We give an algebraic interpretation of the product of resultants for a special choice of the quadratic polynomials and will prove Theorem \ref{thm: main thm} by reducing to this case.

\begin{proposition}
\label{prop:res=det}
Let $Q_1$, $Q_2$ and $Q_3$ be homogeneous degree 2 polynomials in $u$ and $v$ and let $P_1=Q_2Q_3$, $P_2=Q_1Q_3$ and $P_3=Q_1Q_2$.
Then
\begin{equation} 
\label{eq: det = prod res}
\det A_{P_1,P_2,P_3}=\Res(Q_1,Q_2)\Res(Q_2,Q_3)\Res(Q_1,Q_3)
\end{equation}
where $A_{P_1,P_2,P_3}$ is the matrix defined in Proposition \ref{prop: local index algebraically}.
\end{proposition}
\begin{proof}
The Proposition can be shown by computing both sides of \eqref{eq: det = prod res}. However, we want to give another proof that illustrates what is going on.

We first show that $\det A_{P_1,P_2,P_3}=0$ if and only if $\Res(Q_1,Q_2)\Res(Q_2,Q_3)\Res(Q_1,Q_3)=0$. As both sides in \eqref{eq: det = prod res} are homogeneous polynomials in the coefficients of the $Q_i$, it follows that one is a scalar multiple of the other. We show equality for nice choice of $Q_1$, $Q_2$ and $Q_3$ which implies equality in \eqref{eq: det = prod res}. 

Assume that $\Res(Q_1,Q_2)=0$. Then $Q_1$ and $Q_2$ have a common degree 1 factor and there are $r_1,r_2\in k[u,v]_1$ such that $r_1Q_1+r_2Q_2=0$. Thus $r_2P_1+r_1P_2+0\cdot P_3=Q_3(r_1Q_1+r_2Q_2)=0$. So $\det A_{P_1,P_2,P_3}=0$ by Lemma \ref{lemma: detA=0}.

Conversely, assume that $\det A_{P_1,P_2,P_3}=0$ and there are $r_1,r_2,r_3\in k[u,v]_1$ not all zero, such that $r_1P_1+r_2P_2+r_3P_3=0$. 
Assume $r_1\neq 0$. Then
\begin{align*}
&0=r_1P_1+r_2P_2+r_3P_3=r_1Q_2Q_3+r_2Q_1Q_3+r_3Q_1Q_2\\
\Leftrightarrow &-r_1Q_2Q_3=Q_1(r_2Q_3+r_3Q_2).
\end{align*}
Since $\deg r_1=1$ and $\deg Q_1=2$, either $Q_2$ and $Q_1$ share a degree 1 factor or $Q_3$ and $Q_1$ do (or both) and thus $\Res(Q_1,Q_2)=0$ or $\Res(Q_1,Q_3)=0$.

To show equality of \eqref{eq: det = prod res} it remains to show equality for one choice of $Q_1$, $Q_2$ and $Q_3$ such that $\det A_{P_1,P_2,P_3}$ and $\Res(Q_1,Q_2)$, $\Res(Q_1,Q_3)$, $\Res(Q_2,Q_3)$ are all nonzero.
Let $Q_1=u^2$, $Q_2=u^2+v^2$ and $Q_3=v^2$. Then $\Res(Q_1,Q_2)=\Res(Q_1,Q_3)=\Res(Q_2,Q_3)=1$ and $\det A_{P_1,P_2,P_3}=1$.
\end{proof}

\subsection{Proof of Theorem \ref{thm: main thm}}
Let $l$ be a simple and isolated line with field of definition $k$ on a quintic threefold and let $C=(P_1:P_2:P_3):l\rightarrow\PP^2$ be the associated degree 4 curve defined in section \ref{section: defn of type}. 
There are the following possibilities for how $C:l\rightarrow \PP^2$ looks like.
\begin{enumerate}
\item $C$ is birational onto its image. Then there are two possibilities
\begin{enumerate}
\item the generic case: the image has three distinct double points in general position.
\item the curve has a tacnode, that is a double point with multiplicity two and another double point.
\end{enumerate}
\item $C$ is not birational onto its image. 
\end{enumerate}
A degree $4$ map $\PP^1\rightarrow \PP^4$ which is not birational onto its image, is either a degree $4$ cover of a line or a degree $2$ cover of a conic. It was pointed out to me by the anonymous referee that if $C$ were a degree $4$ cover of a line, there would be linear $r_1,r_2,r_3\in k[u,v]$ such that $r_1P_1+r_2P_2+r_3P_3=0$. In this case $\det A_{P_1,P_2,P_3}=0$ as shown in Proposition \ref{prop:res=det}, and the line line would not be simple and isolated.
So in case $C$ is not birational onto its images, it is a degree $2$ cover of a conic.

We prove Theorem \ref{thm: main thm} case by case.
\subsubsection{The generic case}
In case $C$ has three double points in general position, we perform a coordinate change of $\mathbb{P}^4$ with the aim to being able to apply Proposition \ref{prop:res=det}. We do this for all possible field extensions over which the three double points in $\mathcal{M}$ can be defined. There are three possibilities.
\begin{enumerate}
    \item All three points in $\mathcal{M}$ are $k$-rational.
    \item One point in $\M$ is rational and the other two are defined over a quadratic field extension $L=k(\sqrt{\beta})$. Since $\operatorname{char}k\neq 2$, the field extension $L/k$ is Galois. Let $\sigma$ be the nontrivial element of $\operatorname{Gal}(L/k)$. 
    \item  Let $L=k(\beta)$ be a degree 3 field extension of $k$ and assume that $L$ the is residue field of one of the double points $M$.
Because we assumed that the three double points are in general position and thus pairwise different, we know that $L/k$ is separable. 
Let $E/L$ be the smallest field extension such that $E/k$ is Galois. Let $G$ be the Galois group of $E$ over $k$ and let $\sigma,\tau\in G$ such that the other two double points of $C$ are $\sigma(M)$ and $\tau(M)$.
\end{enumerate}

We choose three special points $M_1$, $M_2$ and $M_3$ in $\PP^2$ for each of the three cases as in the following table.

\begin{center}
\begin{tabular}{ |c|c|c|c| }
 \hline 
 & 1. & 2. & 3.\\
 \hline
$M_1:=$ & $(1:0:0)$ & $(0:1:0)$ & $(1:\beta:\beta^2)$ \\ 
$M_2:=$ & $(0:1:0)$ & $(0:1:\sqrt{\beta})$ & $(1:\sigma(\beta):\sigma(\beta)^2)$ \\ 
$M_3:=$ & $(0:0:1)$ & $(0:1:-\sqrt{\beta})$ & $(1:\tau(\beta):\tau(\beta)^2)$ \\ 
 \hline
\end{tabular}
\end{center}

\begin{claim}
For each of the three casese, there is $\phi\in \operatorname{Gl}_3(k)$ that maps the three double points to $M_1$, $M_2$ and $M_3$.
\end{claim}
\begin{proof}
\begin{enumerate}
	\item Let $(m_1:m_2:m_3)$, $(n_1:n_2:n_3)$ and $(q_1:q_2:q_3)$ be the three double points. In the first case, they are all $k$-rational and in general position. So the matrix 
$\begin{pmatrix}m_1 &n_1&q_1\\m_2 &n_2&q_2\\m_3 &n_3&q_3 \end{pmatrix}$ has coefficients in $k$ and is invertible. Its inverse is the map we are looking for. 
	\item Let $(m_1:m_2:m_3)$, $(n_1+\sqrt{\beta}q_1:n_2+\sqrt{\beta}q_2:n_3+\sqrt{\beta}q_3)$ and $(n_1-\sqrt{\beta}q_1:n_2-\sqrt{\beta}q_2:n_3-\sqrt{\beta}q_3)$ be the three double points with
$m_i,n_i,q_i\in k$ for $i=1,2,3$. Again the matrix $\begin{pmatrix}m_1 &n_1&q_1\\m_2 &n_2&q_2\\m_3 &n_3&q_3 \end{pmatrix}$ is invertible and its inverse sends the three double points to $M_1$, $M_2$ and $M_3$.
	\item Let $(m_1+\beta n_1+\beta^2 q_1 : m_2+\beta n_2+\beta^2 q_2: m_3+\beta n_3+\beta^2 q_3)$, $(m_1+\sigma(\beta) n_1+\sigma(\beta)^2 q_1:m_2+\sigma(\beta) n_2+\sigma(\beta)^2 q_2: m_3+\sigma(\beta) n_3+\sigma(\beta)^2 q_3)$ and $(m_1+\tau(\beta) n_1+\tau(\beta)^2 q_1:m_2+\tau(\beta) n_2+\tau(\beta)^2 q_2: m_3+\tau(\beta) n_3+\tau(\beta)^2 q_3)$ be the three double points. Again the map we are looking for is the inverse of $\begin{pmatrix}m_1 &n_1&q_1\\m_2 &n_2&q_2\\m_3 &n_3&q_3 \end{pmatrix}$.
\end{enumerate}
\end{proof}

Let $\phi=\begin{pmatrix}
a_1 & b_1 & c_1 \\
a_2 & b_2 & c_2 \\
a_3 & b_3 & c_3 
\end{pmatrix}$ 
be the endomorphism of $\PP^2$ that maps the three double points to $M_1$, $M_2$ and $M_3$.
We replace the coordinates $x_1,x_2,x_3,u,v$ of $\PP^4$ by $a_1x_1+a_2x_2+a_3x_3,b_1x_1+b_2x_2+b_3x_3,c_1x_1+c_2x_2+c_3x_3,u,v$. Then 
\begin{align*}
f&=(a_1x_1+a_2x_2+a_3x_3)P_1(u,v)+(b_1x_1+b_2x_2+b_3x_3)P_2(u,v)+(c_1x_1+c_2x_2+c_3x_3)P_3+Q'\\
&=x_1(a_1P_1+b_1P_2+c_1P_3)+x_2(a_2P_1+b_2P_2+c_2P_3)+x_3(a_3P_1+b_3P_2+c_3P_3)+Q'\\
&=:x_1P_1'+x_2P_2'+x_3P_3'+Q'
\end{align*}
for some $Q'\in (x_1,x_2,x_3)^2$ and $P_1':=a_1P_1+b_1P_2+c_1P_3$, $P_2':=a_2P_1+b_2P_2+c_2P_3$ and $P_3':=a_3P_1+b_3P_2+c_3P_3$.
Let $C':=(P_1':P_2':P_3'):l\rightarrow \PP^2$. Then the three double points of $C'$ are $M_1$, $M_2$ and $M_3$. 
\begin{claim}After a coordinate change we can assume that $P_1$, $P_2$ and $P_3$ equal

\begin{center}
\begin{tabular}{ |c|c|c|c| }
 \hline 
 & 1. & 2. & 3.\\
 \hline
$P_1=$ & $Q_2Q_3$ & $Q\sigma(Q)$  & $Q\sigma(Q)+\sigma(Q)\tau(Q)+Q\tau(Q)$ \\ 
$P_2=$ & $Q_1Q_3$ & $Q_1\frac{Q-\sigma(Q)}{2\sqrt{\beta}}$ & $\tau(\beta)Q\sigma(Q)+\beta\sigma(Q)\tau(Q)+\sigma(\beta)Q\tau(Q)$ \\ 
$P_3=$ & $Q_1Q_2$ & $Q_1\frac{Q+\sigma(Q)}{2}$ & $\tau(\beta)^2Q\sigma(Q)+\beta^2\sigma(Q)\tau(Q)+\sigma(\beta)^2Q\tau(Q)$ \\ 
 \hline
\end{tabular}
\end{center}
for some homogeneous degree $2$ polynomials $Q,Q_1,Q_2,Q_3\in k[u,v]_2$.
\end{claim}
\begin{proof}
Let $s_i,r_i\in l\cong \PP^1$ be the two points that are sent to $M_i$ by $C'$ and let $Q_i\in k[u,v]_2$ be a homogeneous degree $2$ polynomial with zeros $s_i$ and $r_i$ for $i=1,2,3$.
\begin{enumerate}
	\item Since $C'(s_2)=C'(r_2)=(0:1:0)$ and $C'(s_3)=C'(r_3)=(0:0:1)$, we get that $P'_1(s_j)=P'_1(r_j)=0$ for $j=2,3$. Hence, $Q_2$ and $Q_3$ both divide $P'_1$. We have seen in the proof of Lemma \ref{lemma: base point free} that if $Q_2$ and $Q_3$ had a common factor, then the line $l$ would not be a simple line. So up to scalars in $k^{\times}$, we have that $P'_1=Q_2Q_3$.  Note we can always scale the $P_i'$ because we can replace $x_i$ by $\lambda x_i$ for $\lambda \in k^{\times}$ and $i\in \{1,2,3\}$. Similarly, $P'_2=Q_1Q_3$ and $P'_3=Q_1Q_2$.
	\item Again we know that both $Q_2$ and $Q_3$ divide $P'_1$. Since $P'_1$ has coefficients in $k$, we know that $Q_3=\sigma(Q_2)=:Q$ and $P_1'=Q\sigma(Q)$ up to a scalar in $k^{\times}$. 
Since $\sqrt{\beta}P'_2(s_2)=P'_3(s_2)$ as well as $\sqrt{\beta}P'_2(r_2)=P'_3(r_2)$ we get that $\sigma(Q)$ divides $P_3'-\sqrt{\beta}P_2'$ and since $-\sqrt{\beta}P'_2(s_3)=P'_3(s_3)$ and $-\sqrt{\beta}P'_2(r_3)=P'_3(r_3)$, we know that $Q$ divides $P_3'+\sqrt{\beta}P_2'$. 
Furthermore, $Q_1$ divides $P'_2$ and $P'_3$. Since $P'_2$ and $P'_3$ have coefficients in $k$, it follows that $\sigma(Q)=\frac{P_3'-\sqrt{\beta}P_2'}{Q_1}$ and $Q=\frac{P_3'+\sqrt{\beta}P_2'}{Q_1}$ up to a scalar in $k^{\times}$ and equivalently 
$P_2'=Q_1\frac{Q-\sigma(Q)}{2\sqrt{\beta}}$ and $P_3'=Q_1\frac{Q+\sigma(Q)}{2}$.
	\item We can assume that $Q_1$, $Q_2$ and $Q_3$ are Galois conjugates, so let $Q:=Q_1$, then $Q_2=\sigma(Q)$ and $Q_3=\tau(Q)$. The two zeros $r_1$ and $s_1$ of $Q$ are zeros of \[\beta\sigma(\beta)(\beta-\sigma(\beta))P'_1+(\sigma(\beta)^2-\beta^2)P_2'+(\beta-\sigma(\beta))P_3'\] since $\beta\sigma(\beta)(\beta-\sigma(\beta))P'_1(r_1)+(\sigma(\beta)^2-\beta^2)P_2'(r_1)+(\beta-\sigma(\beta))P_3'(r_1)$ and $\beta\sigma(\beta)(\beta-\sigma(\beta))P'_1(s_1)+(\sigma(\beta)^2-\beta^2)P_2'(s_1)+(\beta-\sigma(\beta))P_3'(s_1)$ equal 
\[
	\beta\sigma(\beta)(\beta-\sigma(\beta)))\cdot 1+(\sigma(\beta)^2-\beta^2)\cdot \beta+(\beta-\sigma(\beta))\cdot \beta^2=0.
\]
Similarly, one sees that the two zeros $r_2$ and $s_2$ of $\sigma(Q)$ are zeros of $\beta\sigma(\beta)(\beta-\sigma(\beta))P'_1+(\sigma(\beta)^2-\beta^2)P_2'+(\beta-\sigma(\beta))P_3'$. Thus we can assume that 
\[Q\sigma(Q)=\frac{\beta\sigma(\beta)(\beta-\sigma(\beta))P'_1+(\sigma(\beta)^2-\beta^2)P_2'+(\beta-\sigma(\beta))P_3'}{(\beta-\sigma(\beta))(\sigma(\beta)-\tau(\beta))(\tau(\beta)-\beta)}.\]
And similarly, we get
\[\sigma(Q)\tau(Q)=\frac{\sigma(\beta)\tau(\beta)(\sigma(\beta)-\tau(\beta))P'_1+(\tau(\beta)^2-\sigma(\beta)^2)P_2'+(\sigma(\beta)-\tau(\beta))P_3'}{(\beta-\sigma(\beta))(\sigma(\beta)-\tau(\beta))(\tau(\beta)-\beta)}\]
and
\[\tau(Q)Q=\frac{\tau(\beta)\beta(\tau(\beta)-\beta)P'_1+(\beta^2-\tau(\beta)^2)P_2'+(\tau(\beta)-\beta)P_3'}{(\beta-\sigma(\beta))(\sigma(\beta)-\tau(\beta))(\tau(\beta)-\beta)}.\]
Equivalently,
\[P_1'=Q\sigma(Q)+\sigma(Q)\tau(Q)+Q\tau(Q),\]\[P_2'=\tau(\beta)Q\sigma(Q)+\beta\sigma(Q)\tau(Q)+\sigma(\beta)Q\tau(Q)\]and\[P_3'=\tau(\beta)^2Q\sigma(Q)+\beta^2\sigma(Q)\tau(Q)+\sigma(\beta)^2Q\tau(Q).\]

\end{enumerate}
\end{proof}

Now we can use Proposition \ref{prop:res=det} to prove Theorem \ref{thm: main thm} in the generic case.
\begin{enumerate}
	\item In case all double points are $k$-rational, we have
\[\operatorname{ind}_l\sigma_f=\langle\det A_{P_1,P_2,P_3}\rangle\overset{\text{\ref{prop:res=det}}}=\langle \Res(Q_1,Q_2)\Res(Q_1,Q_3)\Res(Q_2,Q_3)\rangle=\operatorname{Type}(l).\]
	\item In case two of the double points are defined over the quadratic field extension $L=k(\sqrt{\beta})$, we get
\begin{align*}
\det A_{P_1,P_2,P_3}&=\det A_{P_1,-\sqrt{\beta}P_2+P_3,\sqrt{\beta}P_2+P_3}\cdot\frac{1}{4\beta}\\
&\overset{\ref{prop:res=det}}=\Res(Q,\sigma(Q)\Res(Q_1,Q)\Res(Q_1,\sigma(Q)) \cdot 4\beta\\
&=\Res\left(\frac{Q+\sigma(Q)}{2},\frac{Q-\sigma(Q)}{2\sqrt{\beta}}\right)\Res(Q_1,Q)\Res(Q_1,\sigma(Q))\cdot \frac{4\beta}{4 \beta}.
\end{align*}
Note that $\frac{Q+\sigma(Q)}{2}$ and $\frac{Q-\sigma(Q)}{2\sqrt{\beta}}$ both have coefficients in $k$ and thus $\Res\left(\frac{Q+\sigma(Q)}{2},\frac{Q-\sigma(Q)}{2\sqrt{\beta}}\right)$ compute the degree of $i_{M_1}$. Since $\sigma(\Res(Q_1,Q))=\Res(Q_1,\sigma(Q))$
\[\operatorname{ind}_l\sigma_f=\langle\det A_{P_1,P_2,P_3}\rangle=\langle\Res\left(\frac{Q+\sigma(Q)}{2},\frac{Q-\sigma(Q)}{2\sqrt{\beta}}\right)N_{L/k}(\Res(Q_1,Q))\rangle=\operatorname{Type}(l).
\]
	\item If there is no $k$-rational double point. For the degree $3$ field extension $L$ of $k$ over which the first double point is defined, we have
\begin{align*}
\det A_{P_1,P_2,P_3}=&\det A_{\sigma(Q)\tau(Q),Q\sigma(Q),Q\tau(Q)} (\beta-\sigma(\beta))^2(\beta-\tau(\beta))^2(\sigma(\beta)-\tau(\beta))^2\\
\overset{\ref{prop:res=det}}= &\Res(Q,\sigma(Q))\Res(Q,\tau(Q))\Res(\sigma(Q),\tau(Q))(\beta-\sigma(\beta))^2(\beta-\tau(\beta))^2(\sigma(\beta)-\tau(\beta))^2\\
=&\Res(Q+\sigma(Q),\sigma(\beta)Q+\beta\sigma(Q))\frac{1}{(\sigma(\beta)-\beta)^2}\\&
\Res(Q+\tau(Q),\tau(\beta)Q+\beta\tau(Q))\frac{1}{(\tau(\beta)-\beta)^2}\\
&\Res(\sigma(Q)+\tau(Q),\tau(\beta)\sigma(Q)+\sigma(\beta)\tau(Q))
\frac{1}{(\tau(\beta)-\sigma(\beta))^2}\\
&(\beta-\sigma(\beta))^2(\beta-\tau(\beta))^2(\sigma(\beta)-\tau(\beta))^2\\
=&\Res(Q+\sigma(Q),\sigma(\beta)Q+\beta\sigma(Q))\Res(Q+\tau(Q),\tau(\beta)Q+\beta\tau(Q))\\
&\Res(\sigma(Q)+\tau(Q),\tau(\beta)\sigma(Q)+\sigma(\beta)\tau(Q))\\
=&N_{L/k}(\Res(\tau(Q)+\sigma(Q),\sigma(\beta)\tau(Q)+\tau(\beta)\sigma(Q)))
\end{align*}
and 
\[\langle N_{L/k}(\Res(\tau(Q)+\sigma(Q),\sigma(\beta)\tau(Q)+\tau(\beta)\sigma(Q)))\rangle=\operatorname{Type}(l).\]
\end{enumerate}

\subsubsection{$C$ has a tacnode}
Assume that the double points of $C$ are not in general position, that is they lie on a line. Because of B\'{e}zout's theorem, the three double points cannot be distinct in this case. 
Hence, there is a tacnode, that is a double point of multiplitcity $2$. Note that, again by B\'{e}zout, there cannot be a double point of multiplicity $3$. Since one of the double points has multiplicity 2, its field of definition could be non-separable of degree 2. However, we assumed that $\operatorname{char}k\neq 2$, so this cannot be the case.
So both double points are defined over $k$ and we can assume that after a $k$-linear coordinate change the double point of multiplicity $1$ is $M_1=(1:0:0)$ and the double point of multiplicity $2$ is $M_2=(0:1:0)$. Let  $r_1,s_1\in \PP^1$ be the two points that are sent to $M_1$ by $C$ and $Q_1\in k[u,v]_2$ be a homogeneous degree $2$ polynomial that vanishes $r_1$ and $s_1$. Further,  let  $r_2,s_2\in \PP^1$ be the two points that are sent to $M_2$ by $C$ and $Q_2\in k[u,v]_2$ be a homogeneous degree $2$ polynomial that vanishes $r_2$ and $s_2$.
Then $Q_1$ divides $P_2$ and $P_3$ and $Q_2$ divides $P_1$ and $P_3$. Since $M_2$ is a double point of multiplicity $2$, we get that up to scalars in $k^{\times}$ 
\begin{align*}
P_1&=Q_2^2\\
P_2&=Q_1S\\
P_3&=Q_1Q_2
\end{align*}
for some $S\in k[u,v]_2$.

The degree of $i_{M_1}$ is equal to $\Res(S,Q_2)$ and the degree of $i_{M_2}$ is equal to $\Res(Q_1,Q_2)$.
Computing both sides shows that $\det A_{P_1,P_2,P_3}=\Res(Q_1,Q_2)^2\Res(S,Q_2)$ and thus Theorem \ref{thm: main thm} holds.

\subsubsection{$C$ is not birational onto its image}
\label{subsubsection: infinitely many double points}
Finally, we deal with the case that $C$ is not birational onto its image and $C$ is a degree $2$ cover of a conic, that is, $C:l\cong\PP^1\rightarrow \PP^2$ factors through a degree 2 map $(Q_1:Q_2):\PP^1\rightarrow\PP^1$. 

\[
\begin{tikzcd}
l\cong \PP^1\arrow{rr}{C=(P_1:P_2:P_3)}\arrow[swap]{rd}{(Q_1:Q_2)}&&\PP^ 2\\
&\PP^ 1\arrow[swap]{ru}{(R_1:R_2:R_3)}&
\end{tikzcd}
\]
 
\begin{claim} It holds that $\langle\Res(Q_1,Q_2)\rangle=(\langle\Res(Q_1,Q_2)\rangle)^3=\langle\det_{P_1,P_2,P_3}\rangle\in \GW(k)$.
\end{claim}
\begin{proof}
Let $R_1=c_2u^2+c_1uv+c_0v^2$, $R_2=d_2u^2+d_1uv+d_0v^2$ and $R_3=e_2u^2+e_1uv+e_0v^2$. 
A calculation shows that 
\[\det A_{P_1,P_2,P_3}=\Res(Q_1,Q_2)^3(\det N)^2\] 
where 
\[
N=\begin{pmatrix}
c_2 & d_2 & e_2\\
c_1 & d_1 & e_1\\
c_0 & d_0 & e_0
\end{pmatrix}.
\]
Hence, $\langle\det A_{P_1,P_2,P_3}\rangle=\langle\Res(Q_1,Q_2)^3(\det N)^2\rangle=\langle\Res(Q_1,Q_2)\rangle\in \GW(k)$.
\end{proof}
So, $\langle \det A_{P_1,P_2,P_3}\rangle$ equals the degree of the Segre involution corresponding to the degree $2$ map $(Q_1:Q_2):l\rightarrow \PP^1$.

\begin{remark}
In this case, the type of $l$ reminds of Kass and Wickelgren's definition of the type of a line on a cubic surface. For lines on cubic surfaces the Gauss map $(P_1:P_2):l\rightarrow \PP^1$ has degree $2$ and $\operatorname{Type}(l)=\langle \Res(P_1,P_2)\rangle\in \GW(k)$ \cite[$\S3$]{MR4247570}.
\end{remark}
\begin{example}
\label{example: infinitely many double points}
Let $P_1=u^2(u^2+v^2)$, $P_2=u^2v^2$ and $P_3=v^2(u^2+v^2)$. Then $C(t)=C(-t)$ for each $t\in l\cong\PP^1$. Hence, $C$ factors through $(u^2:v^2):\PP^1\rightarrow \PP^1$ and the corresponding Segre involution $\PP^1\rightarrow
 \PP^1$ is given by $t\mapsto -t$ which has degree $1\in k^{\times}/(k^{\times})^2$ and one computes that $\langle\det A_{P_1,P_2,P_3}\rangle=\langle 1\rangle\in\GW(k)$.
\end{example}

\begin{remark}
Let $p$ be a singular point on a hypersurface $Y=\{h=0\}$, such that the gradient $\operatorname{grad} h$ has an isolated zero at $p$. 
In \cite[\S6.3]{MR4245478} Wickelgren and the author show that for a general deformation of $Y$ the $\A^1$-Milnor number at $p$ equals the sum of $\A^1$-Milnor numbers at the singularities which $p$ bifurcates into. 
We can apply the same argument to our situation. 
Let $l$ be an isolated not necessarily simple line on a quintic threefold $X$. If we deform the threefold $X=\{f=0\}\subset \PP^4_k$ to $X_t=\{f+tg=0\}\subset \PP^4_{k[[t]]}$ for $g$ a general homogeneous degree $5$ polynomial, then the local index $\ind_l\sigma_f$ equals the sum of local indices at the lines the line $l$ deforms to by \cite[Theorem 5]{MR4245478}. We expect these deformations of $l$ to be simple with a Gauss map with three double points (in general position) when the deformation is general, in which case $\ind_l\sigma_f$ would equal the sum of types of lines $l$ deforms to.
\end{remark}

\section{The dynamic Euler number and excess intersection}
\label{section: dynamic Euler number}
Let $\pi:E\rightarrow Y$ be a relatively oriented vector bundle of rank $d$ over a smooth proper $d$-dimensional scheme $Y$ over $k$.
We have seen that for a section $\sigma$ with only isolated zeros, the $\A^1$-Euler number $e^{\A^1}(E)$ is the sum of local indices at the finitely many isolated zeros of $\sigma$. However, many `nice' sections $\sigma$ have non-isolated zeros and we have an \emph{excess intersection}. In this section, we will use \emph{dynamic intersection} to express $e^{\A^1}(E)$ as the sum of local contributions of finitely many closed points in $\sigma^{-1}(0)$ that deform with a general deformation of $\sigma$.

Excess intersection of Grothendieck-Witt groups has already been defined and studied by Fasel in \cite{MR2563143} and Euler classes with support were defined in \cite[Definition 5.1]{MR4198841} and further studied in \cite{MR4321205}. Remark 5.20 in \cite{BW} shows that the contribution from a non-isolated zero which is regularly embedded, is the Euler number of a certain excess bundle.

\subsection{Fulton's intersection product}
Classically, we can define the Euler class $e(E,\sigma)$ of a rank $r$ vector bundle $\pi:E\rightarrow Y$ over a $d$-dimensional scheme $Y$ over $\C$ with respect to a section $\sigma$ as the \emph{intersection product} of $\sigma$ by the zero section $s_0$  \cite[Chapter 6]{MR1644323}.
\begin{equation}
\label{eq: Euler class intersection product}
\begin{tikzcd}
\sigma^{-1}(0)\arrow{r}{s_0'}\arrow{d}{\sigma'} & Y\arrow{d}{\sigma}\\
Y\arrow{r}{s_0} & E
\end{tikzcd}
\end{equation}
Let $C=C_{\sigma^{-1}(0)}Y$ be the normal cone to the embedding $s_0':\sigma^{-1}(0)\rightarrow Y$.
By \cite[p.94]{MR1644323} there is a closed embedding $C\hookrightarrow \sigma'^*E$  which defines a class $[C]\in\CH_d(E\vert_{\sigma^{-1}(0)})$.
The \emph{intersection product} of $\sigma$ by the zero section $s_0$ is the image of $[C]$ under the isomorphism $\sigma^*:\CH_d(\sigma'^*E)\rightarrow \CH_{d-r}(\sigma^{-1}(0))$ and we define the \emph{Euler class $e(E,\sigma)$ with respect to $\sigma$} to be this intersection product
\[e(E,\sigma)=\sigma^*[C]\in \CH_{d-r}(\sigma^{-1}(0)).\]
The image of $e(E,\sigma)$ in $\CH_{d-r}(Y)$ under the inclusion $\CH_{d-r}(\sigma^{-1}(0))\rightarrow \CH_{d-r}(Y)$ is independent of the section $\sigma$ and called the \emph{Euler class} $e(E)$ of $\pi:E\rightarrow Y$.

Let $C_1,\dots,C_s$ be the irreducible subvarieties of $C$. Then $[C]=\sum_{i=1}^s m_i[C_i]$ where $m_i$ is the geometric multiplicity of $C_i$ in $C$.
The subvarieties $Z_i=\pi(C_i)$ of $\sigma^{-1}(0)$ are called \emph{distinguished varieties} of the intersection product and 
\begin{equation}
\label{eq: indices of dist varieities classically}
e(E,\sigma)=\sum m_i\alpha_i
\end{equation}
where $\alpha_i=\sigma^*[C_i]\in \CH_{d-r}(Z_i)$. So the intersection product splits up as a sum of cycles supported on the distinguished varieties.

If $r=d$ and $\sigma$ intersects $s_0$ transversally, then $\sigma^{-1}(0)$ consists of the isolated zeros of $\sigma$ which are the distinguished varieties of the intersection product \eqref{eq: Euler class intersection product}. That means, $e(E,\sigma)$ is supported on the isolated zeros of $\sigma$. When $\pi:E\rightarrow Y$ is also relatively oriented and $Y$ is smooth and proper over an arbitrary field $k$ (and still $r=d$), we have seen that the $\A^1$-Euler number $e^{\A^1}(E)$ is equal to the sum of local indices at the isolated zeros. In other words, the $\A^1$-Euler number $e^{\A^1}(E)$ is `supported' on the zeros of a section with only isolated zeros, that is on the distinguished varieties.

Oriented Chow groups $\tCH^i(Y,L)$ were introduced by Barge and Morel \cite{MR1753295} and further studied by Fasel \cite{MR2542148}. They are an `oriented version' of Chow groups which can be defined for a (smooth) scheme $Y$ over any field $k$. Here, $L\rightarrow Y$ is a line bundle and defines a `twist' of the oriented Chow group.  Levine defines an \emph{Euler class with support}
\[e_{Z}(E,\sigma)\in\tCH^d_Z(Y,(\det E)^{-1})\]
in \cite[p.2191]{MR4198841} which is supported on a closed subset $Z\subset Y$. Let $L$ be a 1-dimensional $k$-vector space and denote by $\GW(k,L)$ the Grothendieck group of isometry classes of finite rank non-degenerate symmetric bilinear forms $V\times V\rightarrow L$. Then for a closed point $x\in Y$ it holds that $\tCH_x^d(Y,(\det E)^{-1})=\GW(k(x),(\det E\otimes \det\mathfrak{m}_x/\mathfrak{m}_x^2)^{-1})$ and for $x$ an isolated zero of a section $\sigma$ of the relatively oriented bundle $\pi:E\rightarrow Y$, the local index $\ind_x\sigma$ computes $e_x(E,\sigma)$ as follows.
Let $\psi:U\rightarrow \Spec (k(x)[x_1,\dots,x_n])$ be Nisnevich coordinates around $x$ and $E\vert_U\cong U\times \Spec (k(x)[y_1,\dots, y_d])$ a trivialization compatible with $\psi$ and the relative orientation of $E$. Then 
\[e_x(E,\sigma)=\ind_x\sigma_{k(x)}\otimes y_1\wedge\dots\wedge y_d\otimes \bar{x}_1^*\wedge\dots\wedge\bar{x}_d^*\in \GW(k(x),(\det E\otimes \det\mathfrak{m}_x/\mathfrak{m}_x^2)^{-1})\]
by \cite[$\S5$]{MR4198841} where the $\bar{x}_i$ are the images of the Nisnevich coordinates in $\mathfrak{m}_x/\mathfrak{m}_x^2$. In other word the local index of an isolated zero agrees with this contribution described by Levine. 

We conjecture that for any section $\sigma$ of $\pi:E\rightarrow Y$, not only sections with only isolated zeros, the $\A^1$-Euler number is the sum of `local indices' at the distinguished varieties $Z_i$ as in the classical case \eqref{eq: indices of dist varieities classically}. 
We will see that this is true in the case of the section $\sigma_F$ of $\mathcal{E}=\Sym^5\mathcal{S}^*\rightarrow \Gr(2,5)$ defined by the Fermat quintic threefold $\{F=X_0^5+X_1^5+X_2^5+X_3^5+X_4^5=0\}\subset \PP^4$, that is, we will show that there are well-defined `local indices' at the distinguished varieties of the intersection product of $\sigma_F$ by the zero section $s_0$. We then verify that the sum of these local indices is equal to $e^{\A^1}(\mathcal{E})$ in $\GW(k)$. The local index at a distinguished variety $Z$ can be computed as follows. For each deformation of the Fermat we can assign local indices to the points in $\sigma_F^{-1}(0)\subset \Gr(2,5)$ that deform and the local index at $Z$ is the sum of local indices at points in $Z$ that deform with a general deformation. It turns out that the local index at $Z$ is well-defined, that is, it does not depend on the chosen general deformation of the Fermat.

\subsection{The dynamic Euler number}
One way to find the well-defined zero cycle supported on the distinguished varieties classically is to use \emph{dynamic intersection} \cite[Chapter 11]{MR1644323}. We deform a section $\sigma$ of a $\operatorname{rank}E=\dim Y$ bundle $\pi:E\rightarrow Y$ to $\sigma_t:=\sigma+t\sigma_1+t^2\sigma_2+\dots$ where the $\sigma_i$ are general sections of $E$. The deformation has finitely many isolated zeros and the zero cycle we are looking for is the `limit' $t\rightarrow 0$ of $\sigma_t^{-1}(0)$ which is supported on $\sigma^{-1}(0)$ \cite[Theorem 11.2]{MR1644323}. Moreover, for a general deformation the zero cycle $m_i\alpha_i\in \CH_{d-r}(Z_i)$ from \eqref{eq: indices of dist varieities classically} supported on a distinguished variety $Z_i$ is the part of the limit of $\sigma_t^{-1}(0)$ supported on $Z_i$ \cite[Proposition 11.3]{MR1644323}.

It follows that over the complex numbers the Euler number can be can be computed as the count of zeros of a section (with non-isolated zeros) that deform with a general deformation. For example, Segre finds 27 distinguished lines on the union of three hyperplanes in $\PP^3$ which deform with a general deformation\cite{MR0008171} and 27 is the classical count of lines on a general cubic surface \cite{cayley_2009}. 
Albano and Katz find the limits of 2875 complex lines on a general deformation of the Fermat quintic threefold in \cite{MR1024767}. In section \ref{section: lines on fermat} we will use those 2875 limiting lines to compute the `dynamic Euler number' of $\Sym^5\mathcal{S^*}\rightarrow \Gr(2,5)$ valued in $\GW(k((t)))$.

\subsubsection{$\GW(k((t)))$}
In order to understand the computations in section \ref{section: lines on fermat}, we recall some properties of $\GW(k((t)))$.
Any uni in $k((t))$ is of the form $u=\sum_{i=m}^{\infty}a_it^i$ with $a_m\neq 0$.
One can factor $u$ as $u=a_mt^m(1+\sum_{i=1}^{\infty}b_it^i)$ with $b_i=\frac{a_{i+m}}{a_m}$. 
\begin{claim}
\label{claim: structure of GW(k((t)))}
$1+\sum_{i=1}^{\infty}b_i$ is a square in $k((t))^{\times}$. 
\end{claim}
\begin{proof}
$1+\sum_{i=1}^{\infty}b_i$ is even a square in $k[[t]]^{\times}$ since one can solve inductively for $c_i\in k$ such that $1+\sum_{i=1}^{\infty}b_i=(1+\sum_{i=1}^{\infty}c_i)^2$.
\end{proof}
It follows that
 \[\langle u\rangle=\langle \sum_{i=m}^{\infty}a_it^i \rangle=\langle a_mt^m \rangle= \begin{cases*}
      \langle a_m\rangle & if m \text{ even}\\
      \langle ta_m\rangle        & if m \text{ odd.}
    \end{cases*}\]
    
Note that $\HH:=\langle1\rangle+\langle-1\rangle=\langle t\rangle+\langle-t\rangle$. Claim \ref{claim: structure of GW(k((t)))} illustrates the content of the following theorem.
\begin{theorem}[Springer's Theorem \cite{MR2104929}]
\label{thm: Springer}
\[\frac{\GW(k)\oplus \GW(k)}{\Z(\HH,-\HH)}\xrightarrow{(i,j)}\GW(k((t))) \]
is an isomorphism. Here $i(\langle a\rangle)=\langle a\rangle\in \GW(k((t)))$ and $j(\langle a\rangle)=\langle ta\rangle\in \GW(k((t)))$.
\end{theorem} 

\subsubsection{Definition of the dynamic Euler number}

Let $\pi:E\rightarrow Y$ be a relatively orientable vector bundle with $\operatorname{rank}E=\operatorname{dim} Y$, $Y$ smooth and proper over $k$, and let $\sigma:Y\rightarrow E$ be a section. We deform the section $\sigma$ to $\sigma_t=\sigma+t\sigma_1+t^2\sigma_2+\dots$ for $\sigma_i$ general sections of $E$. Then $\sigma_t$ is a general section of the base change $E_{k((t))}$ to the field $k((t))$. In particular, $\sigma_t$ has only isolated zeros and the $\A^1$-Euler number 
$e^{\A^1}(E_{k((t))})\in \GW(k((t)))$ is equal to the sum of local indices at those isolated zeros.

\begin{definition}
\label{defn: dynamic Euler number}
We call 
\begin{equation}
\label{eq: dyn euler}
e^{\operatorname{dynamic}}(E):=e^{\A^1}(E_{k((t))})=\sum_{x_t\in \sigma_{t}^{-1}(0)}\ind_{x_t}\sigma_t\in \GW(k((t)))
\end{equation}
the \emph{dynamic Euler number} of $E$.
\end{definition}

By functoriality of the Euler class this sum \eqref{eq: dyn euler} is in the image of the injective map $i:\GW(k)\rightarrow\GW(k((t)))$ from Springer's theorem \ref{thm: Springer}. In other words, the $\A^1$-Euler number $e^{\A^1}(E)$ in is the unique element of $\GW(k)$ that is mapped to $e^{\operatorname{dynamic}}(E)$ by $i$. We will see moreover how the local indices at the lines limiting to a distinguished variety of the intersection product of the Fermat section $\sigma_F$ by the zero section sum up to an element of $\GW(k)$ independent of the deformation, even though the local indices at these lines themselves depend on the deformation.


\section{The lines on the Fermat quintic threefold}
\label{section: lines on fermat}
Let $X=\{F=X_0^5+X_1^5+X_2^5+X_3^5+X_4^5=0\}\subset \PP^4$ be the Fermat quintic threefold. It is well known that there are infinitely many lines on $X$. In \cite[$\S1$]{MR1024767} Albano and Katz show that the complex lines on $X$ are precisely the lines that lie in one of the 50 irreducible components $X\cap V(X_i+\zeta X_j)\subset\PP^4$ where $i,j\in\{0,1,2,3,4\}$, $i\neq j$ and $\zeta$ is a 5th root of unity. Their argument remains true for $\bar{k}$-lines when $\operatorname{char}k\neq 2,5$. 
Let $\sigma_F$ be the section of $\mathcal{E}=\Sym^5\mathcal{S}^*\rightarrow \Gr(2,5)$ defined by $F$.
So $\sigma_F^{-1}(0)$ is the union of 50 irreducible components which we denote by $W_i$ for $i=1,\dots,50$.
Furthermore, Albano and Katz study the lines on $X$ which are limits of the 2875 lines on a family of threefolds 
\[X_t=\{F_t=F+tG+t^2H+\dots=0\}\rightarrow \Spec (\C[[t]])\] with $X_0=X$ and find the following. 
\begin{proposition}[Proposition 2.2 + 2.4 in \cite{MR1024767}]
\label{prop: lines mult 2}
For a general deformation $X_t$ there are exactly 10 complex lines in each component $W_i$ that deform with monodromy 2 in direction $t$. 
\end{proposition}
\begin{proposition}[Proposition 2.3 + 2.4 in \cite{MR1024767}]
\label{prop: lines mult 5}
The lines that lie in the intersection of two components $W_i\cap W_j$, $i\neq j$ deform with monodromy 5 in direction $t$ for a general deformation $X_t$.
\end{proposition}
This makes in total $2\cdot 10 \cdot 50 + 375\cdot 5=2875$ complex lines as expected (see e.g. \cite{MR3617981}). 
Let 
\[X_t:=\{F_t:=F+tG+t^2H+\dots\}\subset \Spec (k[[t]]).\]
We will show that their computations work over fields of characteristic not equal to 2 or 5 in Proposition \ref{prop: lt for mult 2} and Proposition \ref{prop: lt for mult 5} and thus we have found all the lines on $X_t\rightarrow \Spec k((t))$.

To compute the dynamic Euler number of $\mathcal{E}=\Sym^5\mathcal{S}^*$, we compute the sum of local indices at the lines on the base change $X_t\otimes k((t))=\{F_t=0\}\subset \PP^4_{k((t))}$, which we expect to contain 2875 lines defined over the algebraic closure of $k((t))$.  
The base change $l_t\otimes k((t))$ of one of the 2875 lines described in Proposition \ref{prop: lt for mult 2} and Proposition \ref{prop: lt for mult 5} lies on $X_t\otimes k((t))$. So we have found all the lines on $X_t\otimes k((t))$. 
 By abuse of notation, we denote $X_t\otimes k((t))$ by $X_t$ and $l_t\otimes k((t))$ by $l_t$.
\begin{equation}
\label{eq: Euler number of deformation of Fermat}
e^{\operatorname{dynamic}}(\mathcal{E})=e^{\A^1}(\mathcal{E}_{k((t))})=\sum_{l_t\text{ line on }X_t}\operatorname{ind}_{l_t}\sigma_{F_t} \in \GW(k((t)))
\end{equation}




\subsection{Multiplicity 5 lines}
\label{subsection: mult 5 lines}
Let $l_0$ be a line that lies in the intersection of two components $W_i\cap W_j$ with $i\neq j$. After base change we can assume that $l_0=\{(u:-u:v:-v:0)\}\subset \PP^5$. 
As in (see \cite[Proposition 2.4]{MR1024767}) we introduce $\Z/5$ monodromy: 
Since we expect $l_0$ to deform to order 5, we expect a deformation $l_t$ of $l_0$ to be a $\Spec L[[t^{1/5}]]$ point of $\{\sigma_{F_t}=0\}\subset \Gr(2,5)_{k[[t]]}$ where $L$ is algebraic over $k$.
In particular, we can replace $t^{\frac{1}{5}}$ by $t$ and consider the deformation $F_t=F+t^5G+t^{10}H+\dots$. 

We will show that $l_0$ deforms when $\operatorname{char}k\neq 2,5$ and there are 5 lines $l_t$ on the family $X_t$ over $\Spec(k[[t]])$ with $l_t\vert_{t=0}=l_0$.

We perform a coordinate change such that $l_0\otimes k((t))=(0:0:0:u:v)\in \PP^5_{k((t))}$, that is, we replace the coordinates $X_0,X_1,X_2,X_3,X_4$ on $\PP^4_{k((t))}$ by $Y_0=X_0+X_1$, $Y_1=X_2+X_3$, $Y_2=X_4$, $Y_3=X_0$, $Y_4=X_2$. Then $F(Y_0,Y_1,Y_2,Y_3,Y_4)=Y_3^5+(Y_0-Y_3)^5+Y_4^5+(Y_1-Y_4)^5+Y_2^5$.

We choose local coordinates $\Spec (k((t))[x,x',y,y',z,z'])\subset \Gr(2,5)_{k((t))}$, parametrizing the lines spanned by $xe_1+ye_2+ze_3+e_4$ and $x'e_1+y'e_2+z'e_3+e_5$ in $k((t))^5$ for the standard basis $(e_1,\dots,e_5)$ of $k((t))^5$.
Note that $l_0$, which is the span of $e_4$ and $e_5$, corresponds to $0\in \Spec (k((t))[x,x',y,y',z,z'])$ in the local coordinates on $\Gr(2,5)$.

The section $\sigma_{F_t}$ of $\mathcal{E}_{k((t))}$ is
locally given by $f_t=((f_t)_1,\dots,(f_t)_6):\A^6_{k((t))}\rightarrow \A^6_{k((t))}$ where $l_0$ is $0\in \A^6_{k((t))}$ with

\begin{align}
\label{eq: mult 5 equations}
\begin{split}
(f_t)_1=&x^5+y^5+z^5-5x^4+10x^3-10x^2+\mathcolorbox{yellow}{5x}+t^5g_1+\dots\\
(f_t)_2=&5x^4x'+5y^4y'+5z^4z'-5y^4-20x^3x'+30x^2x'-20xx'+\mathcolorbox{yellow}{5x'}+t^5g_2+\dots\\
(f_t)_3=&10x^3x'^2+10y^3y'^2+\mathcolorbox{green}{10z^3z'^2}-30x^2x'^2-20y^3y'+10y^3+30xx'^2-10x'^2+t^5g_3+\dots\\
(f_t)_4=&10x^2x'^3+10y^2y'^3+\mathcolorbox{green}{10z^2z'^3}-20xx'^3-30y^2y'^2+10x'^3+30y^2y'-10y^2+t^5g_4+\dots\\
(f_t)_5=&5xx'^4+5yy'^4+5zz'^4-5x'^4-20yy'^3+30yy'^2-20yy'+\mathcolorbox{yellow}{5y}+t^5g_5+\dots\\
(f_t)_6=&x'^5+y'^5+z'^5-5y'^4+10y'^3-10y'^2+\mathcolorbox{yellow}{5y'}+t^5g_6+\dots
\end{split}
\end{align}
\normalsize
in the chosen coordinates and trivialization of $\mathcal{E}\vert_U$ defined in \eqref{eq:basis Sym5}, that is, $(f_t)_1,(f_t)_2,(f_t)_3,(f_t)_4,(f_t)_5,(f_t)_6$ are the coefficients of $u^5,u^4v,u^3v^2,u^2v^3,uv^4,v^5$ of 
$F_t(Y_0=xu+x'v,Y_1=yu+y'v,Y_2=zu+z'v,Y_3=u,Y_4=v)$. Recall that $F_t=F+t^5G+\ldots$. In \eqref{eq: mult 5 equations}, $g_1,\ldots,g_6$ are the coefficients of $u^5,u^4v,u^3v^2,u^2v^3,uv^4,v^5$ of $G(Y_0=xu+x'v,Y_1=yu+y'v,Y_2=zu+z'v,Y_3=u,Y_4=v)$.

\begin{proposition}
\label{prop: lt for mult 5}
Assume $\operatorname{char}k\neq 2,5$.
Let $a=-\frac{g_3(0)}{10}$ and $b=-\frac{g_4(0)}{10}$. For a general deformation we have $ab\neq 0$ and $L:=\frac{k[w]}{(w^5-ab)}$ is a finite \'{e}tale algebra over $k$ with $\dim_kL=5$. Then there
are 5 solutions of the form $l_t=(x_t,x_t',y_t,y_t',z_t,z_t')\in (L[[t]])^6$ with
\begin{align}
\label{eq: mult 5 deformations}
\begin{split}
x_t&=x_1t+x_2t^2+\dots\\
x_t'&=x_1't+x_2't^2+\dots\\
y_t&=y_1t+y_2t^2+\dots\\
y_t'&=y_1't+y_2't^2+\dots\\
z_t&=z_1t+z_2t^2+\dots\\
z_t'&=z_1't+z_2't^2+\dots\\
\end{split}
\end{align}
to $f(l_t)=0$.
In particular, $x_0=x_0'=y_0=y_0'=z_0=z_0'=0$. That means, 0 (and thus the line $l_0$) deforms with multiplicity 5.
\end{proposition}
\begin{remark}
With the notation from Proposition \ref{prop: lt for mult 5}, $a$ and $b$ both not being zero is the condition in \cite[Proposition 2.3]{MR1024767} that $l_0$ deforms.
\end{remark}
\begin{proof}[Proof of Proposition \ref{prop: lt for mult 5}]
Albano and Katz show that $l_0$ deforms to 5 lines over the complex numbers in \cite[Propoisition 2.3 and 2.4]{MR1024767}. 
We redo their proof in our chosen coordinates to see that it remains true over a field $k$ with $\operatorname{char}k\neq 2,5$.

The $t$-term of $(f_t)_1(l_t)=0$ is equal to $5x_1=0$. Since $\operatorname{char} k\neq 5$, it follows that $x_1=0$. The $t^2$-term in $(f_t)_1(l_t)=0$ is equal to $-10x_1^2+5x_2=5x_2=0$ and we get $x_2=0$. Similarly, the $t^3$ and $t^4$-term in $(f_t)_1(l_t)=0$ imply that $x_3=0$ and $x_4=0$, respectively, and the $t^i$-terms for $i=1,2,3,4$ in $(f_t)_2(l_t)=0$, $(f_t)_5(l_t)=0$ and $(f_t)_6(l_t)=0$, yield that $x_1'=x_2'=x_3'=x_4'=y_1=y_2=y_3=y_4=y_1'=y_2'=y_3'=y_4'=0$.

Setting the $t^5$-terms in $(f_t)_3(l_t)$ and $(f_t)_4(l_t)$ equal to zero gives 
\begin{equation}
\label{eq: f3 t5 term}
10z_1^3z_1'^2+g_3(0)=0
\end{equation}
and 
\begin{equation}
\label{eq: f4 t5 term}
10z_1^2z_1'^3+g_4(0)=0.
\end{equation}
For a general deformation $g_3(0)=-10a$ and $g_4(0)=-10 b$ are both not equal to zero. 
Dividing \eqref{eq: f4 t5 term} by \eqref{eq: f3 t5 term} yields
$\frac{z_1'}{z_1}=\frac{b}{a}$ and thus $z_1'=z_1\frac{b}{a}$.
We get 5 solutions for $(z_1,z_1')\in L^2$
corresponding to the 5 solutions of 
\[z_1^3(\frac{b}{a}z_1)^2=\frac{b^2}{a^2}z_1^5=a\Leftrightarrow z_1^5=\frac{a^3}{b^2}.\]

The $t^5$-term in $(f_t)_1(l_t)=0$ is $5x_5+A_1^{(n)}=0$ for a polynomial $A_1^{(n)}$ in $x_1,\dots,x_4,y_1,z_1$ which we have already solved for. So there is one solution for $x_5$ depending on $x_1,\dots,x_4,y_1,z_1$. Similary the $t^5$-terms of $(f_t)_2(l_t)=0$, $(f_t)_5(l_t)=0$ and $(f_t)_6(l_t)=0$ yield unique $x_5'$, $y_5$ and $y_5'$, respectively. 

We show that we can solve for the remaining terms in \eqref{eq: mult 5 deformations} uniquely by induction.
For $n\ge 6$ we show that the $t^n$-term in $f(l_t)=0$ yields unique $x_n,x_n',y_n,y_n',z_{n-4},z'_{n-4}$ assuming that we have already solved for $x_1,\dots,x_{n-1},x_1',\dots,x_{n-1}',y_1,\dots,y_{n-1},y_1',\dots,y_{n-1}',z_1,\dots,z_{n-5},z_1',\dots,z_{n-5}'$.

The $t^n$-term in $(f_t)_1(l_t)$ is equal to $5x_n+A^{(n)}_1=0$ where $A_1^{(n)}$ is a polynomial in \\$x_1,\dots,x_{n-1},x_1',\dots,x_{n-1}',y_1,\dots,y_{n-1},y_1',\dots,y_{n-1}',z_1,\dots,z_{n-5},z_1',\dots,z_{n-5}'$. Hence, there is a unique solution for $x_n$. Similarly, the $t^n$-terms in $(f_t)_2(l_t)=0$, $(f_t)_5(l_t)=0$ and $(f_t)_6(l_t)=0$ determine unique $x_n'$, $y_n$ and $y_n'$, respectively.

The $t^n$-term in $(f_t)_3(l_t)=0$ is $30z_1^2z_{n-4}z_1'^2+20z_1^3z_1'z_{n-4}'+A_3^{(n)}=0$
and the $t^n$-term in $(f_t)_4(l_t)=0$ is $20z_1z_{n-4}z_1'^3+30z_1^2z_1'^2z_{n-4}'+A_4^{(n)}=0$
for polynomials $A^{(n)}_3$ and $A^{(n)}_4$ in \\
$x_1,\dots,x_{n-1},x_1',\dots,x_{n-1}',y_1,\dots,y_{n-1},y_1',\dots,y_{n-1}',z_1,\dots,z_{n-5},z_1',\dots,z_{n-5}'$.
So we got 2 linear equations in $z_{n-4}$ and $z_{n-4}'$ (even in characteristic 3) and get unique solutions for $z_{n-4}$ and $z_{n-4}'$.

By Artin's approximation theorem \cite{MR232018}, the $l_t$ are algebraic.
\end{proof}
In the proof we computed the following coefficients of $l_t$.
\begin{align}
\label{eq: computation of lt mult 5}
\begin{split}
&x_1=x_2=x_3=x_4=0\text{, }
x_1'=x_2'=x_3'=x_4'=0\text{, }\\
&y_1=y_2=y_3=y_4=0\text{, }
y_1'=y_2'=y_3'=y_4'=0\\
&\text{and }z_1=\sqrt[5]{\frac{a^3}{b^2}}\text{, }
z_1'= z_1\frac{b}{a} =\sqrt[5]{\frac{b^3}{a^2}}
\end{split}
\end{align}

For general deformations, the $l_t$'s in Proposition \ref{prop: lt for mult 5} are pairwise different and simple lines on $X_t$ and the coordinate ring of the closed subscheme of the 5 deformations in $\Gr(2,5)_{k((t))}$ is $L((t))$ which is a finite 5-dimensional \'{e}tale algebra over $k((t))$. So the contribution of the $l_t$ to \eqref{eq: Euler number of deformation of Fermat} is equal to $\Tr_{L((t))/k((t))}(\langle Jf(l_t)\rangle)\in \GW(k((t)))$ by Remark \ref{remark: etale algebra} where $Jf=\det (\operatorname{jacobian} f)$. 


We compute the lowest term of $Jf(l_t)$ which completely determines $\langle Jf(l_t)\rangle\in \GW(L((t)))$ by Claim \ref{claim: structure of GW(k((t)))}. We highlight the relevant terms of the jacobian of $((f_t)_1,\dots,(f_t)_6)$ evaluated at the $l_t$ which contribute to the lowest term of its determinant.

\[
\begin{pmatrix}
\mathcolorbox{yellow}{5}+t^5(\dots) & t^5(\dots) & t^5(\dots) & t^5(\dots) &5z_1^4t^4+t^5(\dots) & t^5(\dots)\\
t^5(\dots) & \mathcolorbox{yellow}{5}+t^5(\dots) & t^5(\dots) & t^5(\dots) & 20z_1^3z_1't^4+t^5(\dots) &5z_1^4t^4+t^5(\dots)\\
t^5(\dots)& t^5(\dots) & t^5(\dots) & t^5(\dots)& \mathcolorbox{green}{30z_1^2z_1'^2t^4}+t^5(\dots) & \mathcolorbox{green}{20z_1^3z_1't^4}+t^5(\dots)\\
t^5(\dots)& t^5(\dots) & t^5(\dots) & t^5(\dots) & \mathcolorbox{green}{20z_1z_1'^3t^4}+t^5(\dots) &\mathcolorbox{green}{30z_1^2z_1'^2t^4}+t^5(\dots)\\
t^5(\dots)& t^5(\dots) &\mathcolorbox{yellow}{5}+t^5(\dots)  &t^5(\dots) & 5z_1'^4t^4+t^5(\dots)&20z_1z_1'^3t^4+t^5(\dots)\\
t^5(\dots)& t^5(\dots) & t^5(\dots) &  \mathcolorbox{yellow}{5}+t^5(\dots) & t^5(\dots) &5z_1'^4t^4+t^5(\dots) 
\end{pmatrix}
\]

\normalsize

The lowest term of $\det( \operatorname{jacobian} ((f_t)_1,\dots,(f_t)_6))(l_t)$ is $\mathcolorbox{yellow}{5^4} \cdot \mathcolorbox{green}{t^8\cdot z_1^4\cdot z_1'^4\cdot(900-400)}$ and thus 
\[\langle Jf(l_t)\rangle=\langle 500\cdot 5^4  t^8 z_1^4z_1'^4\rangle =\langle 5\rangle \in \GW(L((t))).\]
Hence, we get the following contribution of the 5 lines to \eqref{eq: Euler number of deformation of Fermat} 
\begin{equation}
\label{eq: contribution mult 5 lines}
\Tr_{L((t))/k((t))}(\langle Jf(l_t)\rangle)=\Tr_{L((t))/k((t))}(\langle 5\rangle)\in \GW(k((t))).
\end{equation}

\begin{lemma}
\label{lemma:zeta5}
Let $F$ be a field of characteristic not equal to $2$ or $5$ and let $R$ be an $F$-algebra.
Let $A=\frac{R[w]}{(w^5-\alpha)}$ for some $\alpha\in R^{\times}$. Then $A$ is a free $R$-module of rank $5$ and for a unit $r\in R^{\times}$ we have $\Tr_{A/R}(\langle r\rangle)=2(\langle1\rangle+\langle-1\rangle)+\langle5r\rangle$.
\end{lemma}
\begin{proof}
The following is a basis for $A$ as an $R$-module: $1,w,w^2,w^3,w^4$. 
Let $c=c_1+c_2w+c_3w^2+c_4w^3+c_5w^4\in A$ and $d=d_1+d_2w+d_3w^2+d_4w^3+d_5w^4\in A$ for $c_1,\dots,c_5,d_1,\dots,d_5\in R$. Then
\begin{align*}
r\cdot cd&=r(c_1d_1+\alpha(c_2d_5+c_3d_4+c_4d_3+c_5d_2))\\
&+rw(c_1d_2+c_2d_1+\alpha(c_3d_5+c_4d_4+c_5d_3))\\
&+rw^2(c_1d_3+c_2d_2+c_3d_1+\alpha(c_4d_5+c_5d_4))\\
&+rw^3(c_1d_4+c_2d_3+c_3d_2+c_4d_1+\alpha c_5d_5)\\
&+rw^4(c_1d_5+c_2d_4+c_3d_3+c_4d_2+c_5d_1)
\end{align*}
and 
\[\Tr_{A/R}(rcd)=5r(c_1d_1+\alpha(c_2d_5+c_3d_4+c_4d_3+c_5d_2)).\]
Since
\[\begin{pmatrix}\frac{1}{10r\alpha}& 1\\ -\frac{1}{10r\alpha}& 1\end{pmatrix}\cdot \begin{pmatrix}0& 5r\alpha\\ 5r\alpha& 0\end{pmatrix}\cdot\begin{pmatrix}\frac{1}{10r\alpha}& -\frac{1}{10r\alpha}\\1& 1\end{pmatrix}=\begin{pmatrix}1& 0\\0& -1\end{pmatrix}\]
the class of $\Tr_{A/R}(rcd)$ in $\GW(R)$  equals $\langle 5r\rangle+2(\langle1\rangle+\langle-1\rangle).$
\end{proof}
It follows that
\begin{equation}
\label{eq: contribution single mult 5 line}
\Tr_{L((t))/k((t))}(\langle 5\rangle)\in \GW(k((t)))=2(\langle1\rangle+\langle-1\rangle)+\langle5\cdot 5\rangle=2(\langle1\rangle+\langle-1\rangle)+\langle1\rangle.\end{equation}

We now want to compute the contribution of all lines from Proposition \ref{prop: lt for mult 5} to \eqref{eq: Euler number of deformation of Fermat}.
Recall that $\sigma_F^{-1}(0)$ is the union of $50$ irreducible components and note that the union of lines in the intersection of two irreducible components are the lines in
\[\bigcup_{i,j,m,n} X\cap V(X_i^5+X_j^5)\cap V(X_m^5+X_n^5)\]
where the union is over all pairwise different $i,j,m,n\in\{0,1,2,3,4\}$.

Fix $i,j,m,n \in \{0,1,2,3,4\}$ pairwise different. The closed subscheme of lines in $ X\cap V(X_i^5+X_j^5)\cap V(X_m^5+X_n^5)$ of the Grassmannian $\Gr(2,5)$ has coordinate ring isomorphic to $E:=\frac{k[x,y]}{(x^5-1,y^5-1)}$.
There are $15=\frac{\binom{5}{2}\binom{3}{2}}{2}$ choices for $i,j,m,n\in \{0,1,2,3,4\}$.
It follows that the local contribution of all the lines  of two irreducible components to \eqref{eq: Euler number of deformation of Fermat} is 
\[15\cdot \Tr_{E((t))/k((t))}(2(\langle1\rangle+\langle-1\rangle)+\langle1\rangle).\]
Applying Lemma \ref{lemma:zeta5} two more times we get that the sum of the contributions of the lines in Proposition \ref{prop: lt for mult 5} to \eqref{eq: Euler number of deformation of Fermat} is
\begin{equation}
\label{eq: all mult 5 lines}
15\cdot \Tr_{E((t))/k((t))}( 2\cdot (\langle1\rangle+\langle-1\rangle)+\langle1\rangle)=15(25\cdot 2\cdot \HH+12\cdot \HH+\langle 1\rangle)=930\cdot \HH+\langle1\rangle \in \GW(k((t))).
\end{equation}


\subsubsection{Using the local analytic structure}
We want to present a different approach to finding the contribution of the multiplicity 5 lines.
Clemens and Kley find that the \emph{local analytic structure} at the crossings (that is at the intersection of 2 components) is $\frac{\C[z,z']}{(z^3z'^2,z^2z'^3)}$ \cite[Example 4.2]{Clemens1998CountingCW}.
That means that the local ring of $\sigma_F^{-1}(0)$ at a multiplicity 5 line is isomorphic to $\frac{\C[z,z']_{(z,z')}}{(z^3z'^2,z^2z'^3)}$ when $k=\C$. 

Let $(f_1,\ldots,f_6)$ be the coefficients of $u^5,u^4v,u^3v^2,u^2v^3,uv^4,v^5$ of $F(Y_0=xu+x'v,Y_1=yu+y'v,Y_2=zu+z'v,Y_3=u,Y_4=v)$, that is, 
\begin{align*}
f_1=&x^5+y^5+z^5-5x^4+10x^3-10x^2+5x\\
f_2=&5x^4x'+5y^4y'+5z^4z'-5y^4-20x^3x'+30x^2x'-20xx'+5\\
f_3=&10x^3x'^2+10y^3y'^2+10z^3z'^2-30x^2x'^2-20y^3y'+10y^3+30xx'^2-10x'^2\\
f_4=&10x^2x'^3+10y^2y'^3+10z^2z'^3-20xx'^3-30y^2y'^2+10x'^3+30y^2y'-10y^2\\
f_5=&5xx'^4+5yy'^4+5zz'^4-5x'^4-20yy'^3+30yy'^2-20yy'+5y\\
f_6=&x'^5+y'^5+z'^5-5y'^4+10y'^3-10y'^2+5y'
\end{align*}
Observe that $(f_1,f_2,f_5,f_6)$ is a regular sequence, and setting $x=0$, $x'=0$, $y=0$ and $y'=0$ in $f_3$ and $f_4$ yields $10z^3z'^2$ and $10z^2z'^3$. Dividing both polynomials by 10, we get Clemens and Kley's local structure.

When we deform $z^3z'^2$ and $z^2z'^3$ to $z^3z'^2+t^5g_3$ and $z^2z'^3+t^5g_4$, 0 deforms with multiplicity 5. Let $a=-g_3(0)$ and $b=-g_4(0)$. Then the deformations $(z_t,z'_t)$ of 0 are defined over $L((t))$ where $L=\frac{k[w]}{(w^5-ab)}$ and the local $\A^1$-degree at the deformations of zero is

\begin{align*}
&\deg^{\A^1}_{(z_t,z_t')}(z^3z'^2+t^5g_3,z^2z'^3+t^5g_4)\\
=&\Tr_{L((t))/k((t))}(\langle\det \begin{pmatrix}
3z^2z'^2+t^5\frac{\partial g_3}{\partial z} & 2zz'^3+t^5\frac{\partial g_4}{\partial z}\\
2z^3z'+t^5\frac{\partial g_3}{\partial z'} & 3z^2z'^2+t^5\frac{\partial g_4}{\partial z'}
\end{pmatrix}(z_t,z_t')\rangle )\\
=&\Tr_{L((t))/k((t))}\langle5t^4z_1^4z_1'^4\rangle=\Tr_{L((t))/k((t))}\langle 5\rangle=2\HH+\langle1\rangle\in \GW(k((t))).
\end{align*}

So we get the same contribution as in \eqref{eq: contribution single mult 5 line}. However, this does not reprove what has been done above. To find the local $\A^1$-degree it does not suffice to remember the isomorphism class of the local ring. We need a presentation. That means, the order of the polynomials generating the ideal and the coefficients must not be forgotten. It works in this case because the product of the coefficients of the highlighted terms in \eqref{eq: mult 5 equations} is a square.

\subsection{Multiplicity 2 lines}
Let $l_0$ be one of the lines described in Proposition \ref{prop: lines mult 2} on one of the components $W$ of $\sigma_F^{-1}(0)$. Let $L_0$ be the field of definition of $l_0$.
We introduce $\Z/2$ monodromy and consider the deformation $F_t=F+t^2G+t^{4}H+\dots$.

Since $l_0$ does not lie in one of the intersections $W\cap W_i$ (with $W_i\neq W$),  we can assume that $l_0=\{(u:-u:v:av:bv)\}\subset \PP^5$ with $1+a^5+b^5=0$ and $ab\neq 0$. 

Again we perform a coordinate change such that $l_0=(0:0:0:u:v)$ using coordinates $Y_0=X_0+X_1$, $Y_1=X_3-aX_2$, $Y_2=X_4-bX_2$, $Y_3=X_0$ and $Y_4=X_2$ on $\PP^4$. In the new coordinates \[F=Y_3^5+(Y_0-Y_3)^5+Y_4^5+(Y_1+aY_4)^5+(Y_2+bY_4)^5.\]

We choose the same local coordinates as in \ref{subsection: mult 5 lines} around $l_0\otimes L_0((t))$. Then the section $\sigma_{F_t}$ of $\mathcal{E}_{L_0((t))}$ is
locally given by $f=((f_t)_1,\dots,(f_t)_6):\A^6_{L_0((t))}\rightarrow \A^6_{L_0((t))}$ with
\normalsize
\begin{align*}
(f_t)_1=&x^5+y^5+z^5-5x^4+10x^3-10x^2+\mathcolorbox{yellow}{5x}+t^2g_1+\dots\\
(f_t)_2=&5x^4x'+5y^4y'+5z^4z'+5ay^4+5bz^4-20x^3x'+30x^2x'-20xx'+\mathcolorbox{yellow}{5x'}+t^2g_2+\dots\\
(f_t)_3=&10x^3x'^2+10y^3y'^2+10z^3z'^2-30x^2x'^2+20ay^3y'+20bz^3z'+10a^2y^3+10b^2z^3\\
&+30xx'^2-10x'^2+\mathcolorbox{orange}{t^2g_3}+\dots\\
(f_t)_4=&10x^2x'^3+10y^2y'^3+10z^2z'^3-20xx'^3+30ay^2y'^2+30bz^2z'^2+10x'^3+30a^2y^2y'\\
&+30b^2z^2z'+\mathcolorbox{green}{10a^3y^2}+
\mathcolorbox{green}{10b^3z^2}+t^2g_4+\dots\\
(f_t)_5=&5xx'^4+5yy'^4+5zz'^4-5x'^4+20ayy'^3+20bzz'^3+30a^2yy'^2+30b^2zz'^2+20a^3yy'\\
&+20b^3zz'+\mathcolorbox{pink}{5a^4y}+\mathcolorbox{pink}{5b^4z}+t^2g_5+\dots\\
(f_t)_6=&x'^5+y'^5+z'^5+5ay'^4+5bz'^4+10a^2y'^3+10b^2z'^3+10a^3y'^2\\
&+10b^3z'^2+\mathcolorbox{pink}{5a^4y'}+\mathcolorbox{pink}{5b^4z'}+t^2g_6+\dots
\end{align*}
\normalsize
\begin{proposition}
\label{prop: lt for mult 2}
Assume $\operatorname{char}k\neq 2,5$ and let $d:=-\frac{g_4(0)}{10}$. 
For a general deformation we have $d\neq 0$ and there are 2 solutions $l_t=(x_t,x_t',y_t,y_t',z_t,z_t')\in \left(\frac{L_0[w]}{(w^2+abd)}[[t]]\right)^6$ to $f(l_t)=0$ of the form 
\begin{align*}
x_t&=x_1t+x_2t^2+\dots\\
x_t'&=x_1't+x_2't^2+\dots\\
y_t&=y_1t+y_2t^2+\dots\\
y_t'&=y_1't+y_2't^2+\dots\\
z_t&=z_1t+z_2t^2+\dots\\
z_t'&=z_1't+z_2't^2+\dots\\
\end{align*}
\end{proposition}
\begin{proof}
Again this is done in \cite[Proposition 2.2 and 2.4]{MR1024767} over the complex numbers. 

Setting the coefficients of the $t$-terms in $(f_t)_1(l_t)$ and $(f_t)_2(l_t)$ equal to zero, implies that $x_1=x_1'=0$.
The $t^2$-term in $(f_t)_1(l_t)=0$ and $(f_t)_2(l_t)=0$, gives unique solutions for $x_2$ and $x_2'$, respectively.

Setting the coefficient of the $t$-term of $(f_t)_5(l_t)$ equal to 0, we see that $a^4y_1+b^4z_1=0$. To find $y_1$, we set the $t^2$-term of $(f_t)_4(l_t)$ equal to 0.
\begin{align*}
&0=10a^3y_1^2+10b^3z_1^2+g_4(0)=10a^3y_1^2+10b^3z_1^2-10d\\
\Leftrightarrow &d=a^3y_1^2+b^3\frac{a^8}{b^8}y_1^2=y_1^2\frac{a^3}{b^5}(b^5+a^5)\\
\Leftrightarrow &y_1^2=\frac{-d b^5}{a^3}
\end{align*}
where the last equivalence follows from the equality $a^5+b^5+1=0$.
So there are two solutions $y_1=\pm\frac{b^2}{a}\sqrt{-\frac{db}{a}}$. 

The $t^2$-term in $(f_t)_3(l_t)=0$ is $-10x_1'^2+g_3(0)=g_3(0)=0$. So $l_0$ deforms with the deformation if $g_3(0)$ which is the coefficient of $u^3v^2$ of $G\vert_{l_0}$ and thus a quadratic polynomial $Q$ in $a$ and $b$, is equal to zero.

The $t$-term in $(f_t)_6(l_t)=0$ is a linear equation in $y_1'$ and $z_1'$.

The $t^2$-term in $(f_t)_5(l_t)=0$ is a linear equation in $y_2$ and $z_2$ with coefficients determined by $y_1$, $y_1'$, $z_1$ and $z_1'$, and the $t^2$-term in $(f_t)_6(l_t)=0$ is a linear equation in $y_2'$ and $z_2'$ with coefficients determined by $y_1'$ and $z_1'$.
So the $t^1$ and $t^2$-terms in $f(l_t)=0$ 
\begin{enumerate}
\item determine $x_1,x_2,x_1',x_2',y_1,z_1$,
\item give us 1 linear equation in $y_1'$ and $z_1'$,
\item 1 linear equation in $y_2,z_2$ depending on $y_1'$ and $z_1'$
\item and 1 linear equation in $y_2',z_2'$ depending on $y_1'$ and $z_1'$.
\end{enumerate}
This is the base for the following induction. We assume that the $t^i$-terms of $f(l_t)=0$ for $i=1,\dots,n-1$ determine 
\begin{enumerate}
\item (2 unique solutions for) $x_1,\dots,x_{n-1},x_1',\dots,x_{n-1}'$, \\$y_1,\dots,y_{n-2},y_1',\dots,y_{n-3}',z_1,\dots,z_{n-2},z_1,\dots,z_{n-3}'$,
\item give a 1-dimensional solution space for $y_{n-2}',z_{n-2}'$,
\item a 1-dimensional solution space for $y_{n-1},z_{n-1}$ (depending on $y_{n-2}'$ and $z_{n-2}'$)
\item and a 1-dimensional solution space for $y_{n-1}',z_{n-1}'$ (also depending on $y_{n-2}'$ and $z_{n-2}'$).
\end{enumerate} 
We will show that the $t^n$-term in $f(l_t)=0$ determines
\begin{enumerate}
\item unique $x_n,x_n',y_{n-2}',z_{n-2}',y_{n-1},z_{n-1}$,
\item does not affect the 1-dimensonial solution space for $y_{n-1}',z_{n-1}'$,
\item a 1-dimensional solution space for $y_{n},z_{n}$ (depending on $y_{n-1}'$ and $z_{n-1}'$) 
\item and a 1-dimensional solution space for $y_{n}',z_{n}'$ (also depending on $y_{n-2}'$ and $z_{n-2}'$)
\end{enumerate} 
and thus we have 2 solutions to $f(l_t)=0$ which are algebraic by Artin's approximation theorem \cite{MR232018}.

Let $g_{3,y'}$ be the coefficient of $y'$ and $g_{3,z'}$ be the coefficient of $z'$ in $g_3$.  Setting the $t^n$-term in $(f_t)_3(l_t)$ equal to zero we get $g_{3,y'}y'_{n-2}+g_{3,z'}z'_{n-2}+A^{(n)}_3$ where $A^{(n)}_3$ is a polynomial in  $x_1,\dots,x_{n-1},x_1',\dots,x_{n-1}'$, \\$y_1,\dots,y_{n-2},y_1',\dots,y_{n-3}',z_1,\dots,z_{n-2},z_1,\dots,z_{n-3}'$
which we have already solved for by the induction hypothesis.  Recall that we already have a 1-dimensional solution space for $y'_{n-2}$ and $z'_{n-2}$ determined by a linear equation $5a^4y'_{n-2}+5b^4z'_{n-2}+A^{(n-2)}_6=0$ we get from setting the $t^{n-2}$ in $(f_t)_6(l_t)$ equal to zero (where $A^{(n-2)}_6$ is determined by variables we have already solved for). For a general choice of $G$ in $F_t$ and thus general coefficients $g_{3,y'}$ and $g_{3,z'}$, we get that the two linear equations in $y'_{n-2}$ and $z'_{n-2}$ are linearly independent and thus we get unique solutions for $y'_{n-2}$ and $z'_{n-2}$. 


The $t^n$-term in $(f_t)_1(l_t)=0$ equals $5x_n+A_1^{(n)}=0$ where $A_1^{(n)}$ is a polynomial in $x_1,\ldots,x_{n-1}, x_1',\ldots,x_{n-2}',$ $y_1,\ldots,y_{n-2},y_1',\ldots,y_{n-2}'$, determines $x_n$ uniquely. Similarly, the $t^n$-term in $(f_t)_2(l_t)=0$ determines $x_n'$ uniquely. 

The $t^n$-term of $(f_t)_4(l_t)=0$ is $20a^3y_1y_{n-1}+20b^3z_1z_{n-1}+A^{(n)}_{4}=0$ with $A^{(n)}_{4}$ a polynomial in \\$x_1,\dots,x_{n-1},x_1',\dots,x_{n-1}', y_1,\dots,y_{n-2},y_1',\dots,y_{n-2}',z_1,\dots,z_{n-2},z_1,\dots,z_{n-2}'$.
By the induction hypothesis, we have a $1$-dimensional solution space for $y_{n-1}$ and $z_{n-1}$, namely $5a^4y_{n-1}+5b^4z_{n-1}+A^{(n-2)}_{5}=0$, which we get by setting the $t^{n-1}$-term in $(f_t)_5(l_t)$ equal to zero. Here, $A^{(n-2)}_{5}$ is determined by $x_1,\ldots,x_{n-1},  x_1',\ldots,x_{n-2}', $ $y_1,\ldots,y_{n-2},y_1',\ldots,y_{n-2}'$ which we have already solved for. We claim that these two linear equations in $y_{n-1}$ and $z_{n-1}$ are linearly independent and thus yield unique solutions for $y_{n-1}$ and $z_{n-1}$: If the two equations $20a^3y_1y_{n-1}+20b^3z_1z_{n-1}+A^{(n)}_{4}=0$ and $5a^4y_{n-1}+5b^4z_{n-1}+A^{(n-2)}_{5}=0$ were linearly dependent, then this would imply that $by_1=az_1$. Then $a^4y_1+b^4z_1=0$ would imply that $a^5+b^5=0$ but by hypothesis $a^5+b^5=-1$. 

So we have solved for $y_{n-2}',z_{n-2}',x_n,x_n',y_{n-1},z_{n-1}$ and thus shown 1.

Note that the $1$-dimensional solution space for $y'_{n-1}$ and $z'_{n-1}$ one gets from the $t^{n-1}$-term in $(f_t)_6(l_t)=0$ is not affected by the computations above, this shows 2.

Finally, it is easy to see that the $t^n$-term in $(f_t)_5(l_t)$ and $(f_t)_6(l_t)$ determine linear equations in $y_n,z_n$ and $y_n',z_n'$, respectively, showing 3. and 4.

\end{proof}
In the proof we have calculated that
\begin{align}
\begin{split}
x_1=x_1'=0\\
a^4y_1+b^4z_1=0\text{, } a^4y_1'+b^4z_1'=0\\
y_1= \frac{b^2}{a}\sqrt{-\frac{db}{a}}.
\end{split}
\end{align}
\begin{remark}
\label{remark: general cond mult 2 lines}
As in the proof let $Q(a,b):=g_3(0)$ which is a degree 2 polynomial in $a$ and $b$.
The condition that $l_0$ deforms is $d\neq0$ and $Q(a,b)=0$ in \cite[Proposition 2.2]{MR1024767}.
\end{remark}
Let $E=\frac{k[a,b]}{(1+a^5+b^5,Q(a,b))}$. By the proof of Proposition \ref{prop: lt for mult 5} and Remark \ref{remark: general cond mult 2 lines} the closed subscheme of the 10 multiplicity 2 lines on a component $W$ that deform, has coordinate ring $E=\frac{k[a,b]}{(a^5+b^5+1,Q(a,b))}$. Let $L:=\frac{E[w]}{(w^2+abd)}$. 
Then the contribution of the deformations of the 10 double lines on $W$ to \eqref{eq: Euler number of deformation of Fermat} is
$\Tr_{E((t))/k((t))}(\Tr_{L((t))/E((t))}(\langle Jf(l_t)\rangle))$ where $Jf(l_t)$ is again the determinant of the jacobian of $f$ evaluated at the $l_t$.

Again $\langle Jf(l_t)\rangle$ is determined by the lowest term of $Jf(l_t)$, that is the lowest term of the determinant of the matrix
\normalsize

\[
 \begin{pmatrix}
\mathcolorbox{yellow}{5}+t^2(\dots) & t^2(\dots) & t^2(\dots) & t^2(\dots) &t^2(\dots) & t^2(\dots)\\
 t^2(\dots) & \mathcolorbox{yellow}{5}+t^2(\dots) & t^2(\dots) & t^2(\dots) & t^2(\dots) & t^2(\dots)\\
 t^2(\dots)& t^2(\dots) & t^2(\dots) & \mathcolorbox{orange}{At^2}+t^3(\dots)& t^2(\dots) & \mathcolorbox{orange}{Bt^2}+t^3(\dots)\\
t^2(\dots)& t^2(\dots) & \mathcolorbox{green}{20a^3y_1t}+t^2(\dots) & t^2(\dots) & \mathcolorbox{green}{20b^3z_1t}+t^2(\dots) &t^2(\dots)\\
 t^2(\dots)& t^2(\dots) &\mathcolorbox{pink}{5a^4}+t(\dots)  &20a^3y_1t+t^2(\dots) &\mathcolorbox{pink}{5b^4}+t(\dots)  &20b^3z_1t+t^2(\dots) \\
t^2(\dots)& t^2(\dots) & t^2(\dots) & \mathcolorbox{pink}{5a^4}+t(\dots) & t^2(\dots) &\mathcolorbox{pink}{5b^4}+t(\dots) \\
\end{pmatrix}
\]

\normalsize
where $A:=\frac{\partial g3}{\partial y'}(0)\in E$ and $B:=\frac{\partial g3}{\partial z'}(0)\in E$. The lowest term is equal to
\small
\begin{align*}
&\mathcolorbox{yellow}{5^2}\cdot \left(-\mathcolorbox{orange}{At^2}(\mathcolorbox{green}{20a^3y_1t}\cdot \mathcolorbox{pink}{5b^4}\mathcolorbox{pink}{5b^4}-\mathcolorbox{green}{20b^3z_1t}\cdot \mathcolorbox{pink}{5a^4}\mathcolorbox{pink}{5b^4})-\mathcolorbox{orange}{Bt^2}(-\mathcolorbox{green}{20a^3y_1t}\cdot \mathcolorbox{pink}{5a^4}\mathcolorbox{pink}{5b^4}+\mathcolorbox{green}{20b^3z_1t}\cdot \mathcolorbox{pink}{5a^4}\mathcolorbox{pink}{5a^4})\right)\\
=&25\cdot500\cdot a^3b^3(Ab^4-Ba^4)(z_1a-y_1b)\cdot t^3
\end{align*}
\normalsize
and so 
\normalsize
\begin{align*}
\langle Jf(l_t)\rangle&=\langle 5t ab(az_1-by_1)(b^4A-a^4B)\rangle\\&= \langle -5ty_1ab(a^5+b^5)(b^4A-a^4B)\rangle\overset{1+a^5+b^5=0}{=}
\langle 5tab\cdot(b^4A-a^4B)\frac{b^2}{a}\sqrt{-\frac{db}{a}}\rangle\\&=\langle5t(b^4A-a^4B)b\sqrt{-\frac{db}{a}}\rangle
\end{align*}
\normalsize
in $\GW(L((t)))$. Note that $b^4A-a^4B\neq 0$ for a general deformation $X_t$.
The contribution of the 10 double lines is 
\begin{equation}
\label{eq: contribution of the 10 mult 2 lines}
\Tr_{E((t))/k((t))}(\Tr_{L((t))/E((t)))}(\langle5t(bA-aB)b\sqrt{-\frac{db}{a}}\rangle).
\end{equation}

\begin{lemma}
\label{lemma: quadr ext hyperbolic form}
Let $F$ be a field of characteristic not equal to $2$ and let $R$ be an $F$-algebra. For $\alpha\in R^{\times }$ a unit let $A=\frac{R[w]}{(w^2-\alpha)}$. Then $\Tr_{A/R}(\langle \lambda w\rangle)=\langle1\rangle+\langle -1\rangle$ for any unit $\lambda\in R^{\times}$.
\end{lemma}
\begin{proof}
A basis for $A$ is $1, w$. Let $a=a_1+a_2wb=b_1+b_2w\in A$, then
\[\lambda w ab=\lambda w (a_1b_1+\alpha b_1b_2)+\lambda \alpha(a_1b_2+a_2b_1)\]
and $\Tr_{A/R}(\lambda w ab)=2\lambda \alpha (a_1b_2+a_2b_1)$ which represents $\langle1\rangle+\langle-1\rangle$ in $\GW(R)$ since
\[\begin{pmatrix}\frac{1}{4\lambda\alpha}& 1\\ -\frac{1}{4\lambda\alpha}& 1\end{pmatrix}\cdot \begin{pmatrix}0& 2\lambda\alpha\\ 2\lambda\alpha& 0\end{pmatrix}\cdot\begin{pmatrix}\frac{1}{4\lambda\alpha}& -\frac{1}{4\lambda\alpha}\\1& 1\end{pmatrix}=\begin{pmatrix}1& 0\\0& -1\end{pmatrix}.\]
\end{proof}

\begin{lemma}
\label{lemma: trace of hyperbolic form}
For a finite \'{e}tale algebra $A$ over a field $F$, $\Tr_{A/F}(\langle1\rangle+\langle-1\rangle)=\dim_{F}A\cdot\HH$.
\end{lemma}
\begin{proof}
\[\Tr_{A/F}(\langle1\rangle+\langle-1\rangle)=\Tr_{A/F}(\langle1\rangle)+\Tr_{A/F}(\langle-1\rangle)\]
Assume $\Tr_{A/F}(\langle1\rangle)$ be represented by the form $\beta:V\times V\rightarrow F$. Then $\Tr_{A/F}(\langle-1\rangle)$ is represented by $-\beta$. Now the Lemma follows from the equality $\langle a\rangle+\langle-a\rangle=\langle1\rangle+\langle-1\rangle$ in $\GW(F)$ for any $a\in F^{\times}$.
\end{proof}

It follows from Lemma \ref{lemma: quadr ext hyperbolic form} and Lemma \ref{lemma: trace of hyperbolic form} that \eqref{eq: contribution of the 10 mult 2 lines} is equal to
\begin{equation}
\label{eq: contribution mult 2 lines}
\Tr_{E((t))/k((t))}(\langle1\rangle+\langle-1\rangle)=10 \HH\in \GW(k((t))).
\end{equation}
\subsubsection{Using the local analytic structure}

Clemens and Kley also describe the local analytic structure of the lines on $X$ away from the crossings \cite[Example 4.2]{Clemens1998CountingCW}. It is given by 
\[\frac{\C[z,z']}{(z^2)}.\]


When we naively deform deform $(0,z^2)$ in $k[z,z']$ to $(t^2g_3,z^2+t^2g_4)$ for general $g_3,g_4\in k[z,z']$, 0 deforms with multiplicity 2 over $L=\frac{k[w]}{(w^2-d)}$ where $d=-g_4(0)$. The local $\A^1$-degree at the deformed zeros $(z_t,z_t')$ is 
\small
\begin{align*}
&\Tr_{L((t))/k((t))}\langle\det \begin{pmatrix}
t^2\frac{\partial g_3}{\partial z}(z_t,z_t')&2z_t+t^2\frac{\partial g_4}{\partial z}(z_t,z_t')\\
t^2\frac{\partial g_3}{\partial z'}(z_t,z_t')&t^2\frac{\partial g_4}{\partial z'}(z_t,z_t')
\end{pmatrix}\rangle
\\=&\Tr_{L((t))/k((t))}(\langle-2z_1t^3\frac{\partial g_3}{\partial z'}(0)\rangle)=\Tr_{L((t))/k((t))}(\langle-2t\frac{\partial g_3}{\partial z'}(0)\sqrt{d} \rangle)=\HH.
\end{align*}
\normalsize
Again this does not reprove our computation above, but it illustrates what is going on. 

\subsection{The dynamic Euler number of $\mathcal{E}=\Sym^5\mathcal{S}^*\rightarrow \Gr(2,5)$}




We are ready to compute \eqref{eq: Euler number of deformation of Fermat} by summing up the local indices at the deformed lines described in Proposition \ref{prop: lt for mult 2} and Proposition \ref{prop: lt for mult 5}.
Adding $50\cdot$\eqref{eq: contribution mult 2 lines} to \eqref{eq: contribution mult 5 lines} we get the following theorem.
\begin{theorem}
The dynamic Euler number of $\mathcal{E}=\Sym^5\mathcal{S}^*\rightarrow\Gr(2,5)$ is 
\begin{align*}e^{\operatorname{dynamic}}(\mathcal{E})&=50\cdot 10\cdot \HH+930\cdot \HH+15\langle1\rangle\\
&=1445\langle 1\rangle+1430\langle-1\rangle\in \GW(k((t)))\end{align*}
when $\operatorname{char} k\neq 2,5$.
This gives a new proof for the fact that
\[e^{\A^1}(\mathcal{E})=1445\langle 1\rangle+1430\langle-1\rangle\in \GW(k).\]
\end{theorem}
Note that $e^{\operatorname{dynamic}}(\mathcal{E})$ is independent of $t$ as expected.
Combining this theorem with Theorem \ref{thm: main thm}, we get the following.
\begin{corollary}
For $\operatorname{char}k\neq 2,5$ we have
\[\sum\Tr_{k(l)/k}(\operatorname{Type}(l))=1445\langle1\rangle+1430\langle-1\rangle\in \GW(k)\] 
where the sum runs over the lines on a general quintic threefold.
\end{corollary}

\subsection{Local contributions of the distinguished varieties}
\label{subsection: dist var of Fermat}
Albano and Katz's work \cite[Proposition 2.2 and Proposition 2.3]{MR1024767} determine well-defined zero cycles in $\CH_0(\Gr(2,5))$ supported on each of the 50 components $W_i$ and on the 375 non-empty intersections of 2 components $W_i\cap W_j$ with $i\neq j$. This indicates that the distinguished varieties of the intersection product of $\sigma_F$ by $s_0$ should be
\begin{enumerate}
\item the 50 components $W_i$ of $\sigma_F^{-1}(0)$
\item and the 375 non-empty intersections $W_i\cap W_j$, $i\neq j$.
\end{enumerate}
Clemens and Kley show that these are indeed the distinguished varieties using the local analytic structure \cite[Example 4.2]{Clemens1998CountingCW}.

Let $l$ be a line that deforms with a general deformation $F_t$ to lines $l_{t,1},\dots,l_{t,s}$. Then our computation shows that there is a unique element in $\GW(k)$ that is mapped to $\sum_{i=1}^s\ind_{l_{t,i}}\sigma_{F_t}$ by $i:\GW(k)\rightarrow \GW(k((t)))$. We call this unique element \emph{local index} at $l$ with respect to the deformation $F_t$ and denote it by $\ind_l\sigma_{F_t}$. 
The local indices at the lines that deform with $F_t$ are
\begin{enumerate}
\item $\ind_l\sigma_{F_t}=\Tr_{k(l)/k}(\HH)\in \GW(k)$ for a multiplicity $2$ line $l$,
\item $\ind_l\sigma_{F_t}=\Tr_{k(l)/k}(2\HH+\langle1\rangle)\in \GW(k)$ for a multiplicity $5$ line $l$.
\end{enumerate}
Furthermore, there are well-defined \emph{local indices} $\ind_Z\sigma$ at the distinguished varieties $Z$ independent of the chosen general deformation $F_t$.
\begin{enumerate}
\item The local index at $W_i$ is $\ind_{W_i}\sigma_F:=\sum_{l\in W_i\text{ deforms with }F_t}\ind_l\sigma=10\HH\in\GW(k)$.
\item The local index at $W_i\cap W_j=\{l\}$ is $\ind_{W_i\cap W_j}\sigma_F:=\Tr_{k(l)/k}(2\HH+\langle1\rangle)\in \GW(k)$.
\end{enumerate}

Summing up those local indices we get the following.
\begin{theorem}
Assume $\operatorname{char}k\neq2,5$.
The $\A^1$-Euler number $e^{\A^1}(\mathcal{E})$ can be expressed as the sum of local indices at the distinguished varieties of the intersection product of the section $\sigma_F$ defined by the Fermat quintic threefold by $s_0$
\begin{align*}
e^{\A^1}(\mathcal{E})=1445\langle1\rangle+1430\langle-1\rangle=\sum_{i=1}^{50}\ind_{W_i}\sigma_F+\sum_{i\neq j}\ind_{W_i\cap W_j}\sigma_F\in \GW(k).
\end{align*}

Furthermore, $e^{\A^1}(\mathcal{E})$ is equal to
\begin{align*}
e^{\A^1}(\mathcal{E})&=\sum_{l\text{ deforms with }F_t}\ind_l\sigma_{F_t}\\
&=\sum\Tr_{k(l)/k}(\HH)+\sum\Tr_{k(l)/k}(2\HH+\langle1\rangle)\in\GW(k)
\end{align*}
where the first sum runs over the lines that deform with multiplicity 2 and the second over the lines that deform with multiplicity 5 for a general deformation $F_t$ of the Fermat.
\end{theorem}

\section*{Ackknowledgements}
I would like to thank Kirsten Wickelgren for introducing me to the topic and her excellent guidance and feedback on this project. I am also very grateful to Sheldon Katz for his suggestions and explanations.
Furthermore, I would like to thank Jesse Kass, Paul Arne \O stv\ae r, Thomas Brazelton, Stephen McKean and Gard Olav Helle for helpful discussions and I thank Selina Pauli. 
I also want to thank the anonymous referee for many helpful comments and suggestions.
I gratefully acknowledge support by the RCN Frontier Research Group Project no. 250399 “Motivic Hopf Equations.”
This work was partly initiated during my stay at Duke University.
I also gratefully acknowledge the support of the Centre for Advanced Study at the Norwegian Academy of Science and Letters in Oslo, Norway, which funded and hosted the research project “Motivic Geometry."

\bibliographystyle{alpha}
\bibliography{linebib}
\end{document}